\documentclass[reqno,11pt]{amsart}
\usepackage{amsmath,amsthm,amssymb}
\usepackage{mathtools}
\usepackage{multirow}
\usepackage[utf8]{inputenc}
\usepackage[margin=1in]{geometry}
\usepackage{enumitem}
\usepackage{url}
\usepackage{tikz}
\usepackage{ytableau}
\usepackage{hyperref}
\usepackage[noabbrev,capitalize]{cleveref}
\usepackage{caption}
\usepackage{subcaption}
\usepackage{graphicx}
\crefname{equation}{}{}

\numberwithin{equation}{section}

\newtheorem{theorem}{Theorem}[section]
\newtheorem{lemma}[theorem]{Lemma}
\newtheorem{proposition}[theorem]{Proposition}
\newtheorem{conjecture}[theorem]{Conjecture}
\newtheorem{question}[theorem]{Question}
\newtheorem{corollary}[theorem]{Corollary}

\theoremstyle{definition}
\newtheorem{definition}[theorem]{Definition}
\newtheorem{example}[theorem]{Example}
\newtheorem{remark}[theorem]{Remark}

\makeatletter
\renewcommand\subsection{\@startsection{subsection}{2}%
  \z@{-.5\linespacing\@plus-.7\linespacing}{.5\linespacing}%
  {\normalfont\bfseries}}
\renewcommand\subsubsection{\@startsection{subsubsection}{3}%
  \z@{.5\linespacing\@plus.7\linespacing}{.5\linespacing}%
  {\normalfont\scshape}}
\makeatother

\title{Classical and Consecutive Pattern Avoidance in Rooted Forests}

\author{Swapnil Garg}
\address{Massachusetts Institute of Technology, Cambridge, MA 02139, USA}
\email{swapnilg@mit.edu}

\author{Alan Peng}
\address{Massachusetts Institute of Technology, Cambridge, MA 02139, USA}
\email{apeng1@mit.edu}

\date{May 2020}
\begin{document}
\maketitle
\begin{abstract}
Following Anders and Archer, we say that an unordered rooted labeled forest avoids the pattern $\sigma\in\mathcal{S}_k$ if in each tree, each sequence of labels along the shortest path from the root to a vertex does not contain a subsequence with the same relative order as $\sigma$. For each permutation $\sigma\in\mathcal{S}_{k-2}$, we construct a bijection between $n$-vertex forests avoiding $(\sigma)(k-1)k\coloneqq\sigma(1)\cdots\sigma(k-2)(k-1)k$ and $n$-vertex forests avoiding $(\sigma)k(k-1)\coloneqq\sigma(1)\cdots\sigma(k-2)k(k-1)$, giving a common generalization of results of West on permutations and Anders--Archer on forests. We further define a new object, the forest-Young diagram, which we use to extend the notion of shape-Wilf equivalence to forests. In particular, this allows us to generalize the above result to a bijection between forests avoiding $\{(\sigma_1)k(k-1), (\sigma_2)k(k-1), \dots, (\sigma_\ell)k(k-1)\}$ and forests avoiding $\{(\sigma_1)(k-1)k, (\sigma_2)(k-1)k, \dots, (\sigma_\ell)(k-1)k\}$ for $\sigma_1, \dots, \sigma_\ell \in \mathcal{S}_{k-2}$. Furthermore, we give recurrences enumerating the forests avoiding $\{123\cdots k\}$, $\{213\}$, and other sets of patterns. Finally, we extend the Goulden--Jackson cluster method to study consecutive pattern avoidance in rooted trees as defined by Anders and Archer. Using the generalized cluster method, we prove that if two length-$k$ patterns are strong-c-forest-Wilf equivalent, then up to complementation, the two patterns must start with the same number. We also prove the surprising result that the patterns $1324$ and $1423$ are strong-c-forest-Wilf equivalent, even though they are not c-Wilf equivalent with respect to permutations.
\end{abstract}

\section{Introduction}
In this paper we investigate both classical and consecutive pattern avoidance in rooted forests. A permutation $\pi$ is said to avoid a pattern $\sigma$, another permutation, if no subsequence of $\pi$ has its elements in the same relative order as $\sigma$. Pattern avoidance was first introduced by Knuth in \cite{knuth} to investigate the stack-sorting map, but it has since been generalized to apply to many non-linear objects, including various kinds of trees. We study pattern avoidance on unordered (i.e., non-planar) rooted labeled forests, a notion recently introduced by Anders and Archer in \cite{andersarcher}. Though general pattern avoidance in unordered forests was first explored in \cite{andersarcher}, specific cases have been previously studied. Much research has been done on increasing trees, i.e., trees avoiding the pattern $21$ \cite{increasing1, increasing2, increasing3}. The more general structure of posets, of which unordered forests are a specific example, was studied in a pattern avoidance context by Hopkins and Weiler \cite{poset}. Furthermore, binary trees avoiding a given binary tree structure were studied by Rowland in \cite{rowland}.

The study of classical pattern avoidance has yielded many results in enumerative combinatorics. Two sets of patterns are said to be \textit{Wilf equivalent} if for all positive integers $n$, the number of length-$n$ permutations avoiding the first set is the same as the number of length-$n$ permutations avoiding the second set. Two individual patterns are \textit{Wilf equivalent} if they are Wilf equivalent as single-pattern sets. For example, the number of length-$n$ permutations avoiding a length-$3$ pattern, such as $123$, is the $n$th Catalan number $C_n=\frac{1}{n+1}\binom{2n}{n}$, a fact shown for instance in \cite{simionschmidt}. So, all patterns of length $3$ are Wilf equivalent. Nonrecursive formulas for the number of permutations avoiding a single length-$4$ pattern are known except for those patterns in the Wilf equivalence class of $1324$. For this pattern, a variable-dimension recurrence due to Marinov and Radoi\v ci\'c is given in \cite{rec1324}. Numerous nontrivial Wilf equivalences have been discovered, including the Wilf equivalence between the single-pattern sets $123$ and $132$ through the Simion--Schmidt bijection, which was generalized by J. West to longer patterns \cite{westthesis}, and then generalized further by Backelin, West, and Xin \cite{bwx}.

Work has also been done on pattern avoidance in rooted forests. Anders and Archer enumerated forests avoiding certain sets of patterns, mostly consisting of length-$3$ patterns. They also defined the notion of \textit{forest-Wilf equivalence}, a generalization of Wilf equivalence, and studied forest-Wilf equivalences between certain sets of patterns. They proved the following theorem, generalizing the Simion--Schmidt bijection:
\begin{theorem}[{\cite[Theorem~2]{andersarcher}}]\label{thm:anders-archer}
The patterns $321$ and $312$ are forest-Wilf equivalent.
\end{theorem}
By going into the details of the proof of this result, one may recover the Simion--Schmidt bijection by first applying the proof to the special case when the forest is a path, and then taking the complement; as we will see, taking complements preserves forest-Wilf equivalences as well as Wilf equivalences. This theorem happens to also be a special case of a result proved by Hopkins and Weiler for posets. In this paper, we derive (fixed-dimensional) recurrences for the number of forests avoiding various sets of length-$3$ patterns, including the single patterns $123$ and $213$.

Two sets of patterns are \textit{forest-structure-Wilf equivalent} if for any unlabeled forest structure, the number of labelings avoiding the first set equals the number of labelings avoiding the second set. This is a stronger form of forest-Wilf equivalence. We generalize \cref{thm:anders-archer} to longer patterns, in the vein of West, proving the following:
\begin{theorem}\label{thm:westbijection}
For $k\ge 3$ and $\tau\in\mathcal{S}_k$ such that $\tau(k-1)=k-1$ and $\tau(k)=k$, define $\widetilde\tau\in\mathcal{S}_k$ by letting $\widetilde\tau(i)=\tau(i)$ for $1\le i\le k-2$, $\widetilde\tau(k-1)=k$, and $\widetilde\tau(k)=k-1$. Given a fixed rooted forest, there exists a bijection between labelings avoiding $\tau$ and labelings avoiding $\widetilde\tau$. Thus, the patterns $\tau$ and $\widetilde\tau$ are forest-structure-Wilf equivalent, and therefore forest-Wilf equivalent.
\end{theorem}

We also define a new object called the \textit{forest-Young diagram}, and introduce a notion of Wilf equivalence for this object, called \textit{forest-shape-Wilf equivalence}, analogous to the shape-Wilf equivalence for Young diagrams described in \cite{bwx}. We prove the following theorem, which generalizes further to forest-Young diagrams and sets of patterns:

\begin{theorem}\label{thm:fswe}
Let $m$ be a positive integer. For each integer $i$ with $1\le i\le m$, let $\tau_i\in\mathcal{S}_{k_i}$ be a pattern such that $\tau_i(k_i-1)=k_i-1$ and $\tau_i(k_i)=k_i$. For each such $i$, define $\widetilde\tau_i\in\mathcal{S}_{k_i}$ by letting $\widetilde\tau_i(j)=\tau_i(j)$ for $1\le j\le k_i-2$, $\widetilde\tau_i(k_i-1)=k_i$, and $\widetilde\tau_i(k_i)=k_i-1$. Then, the sets $\{\tau_1,\dots,\tau_m\}$ and $\{\widetilde\tau_1,\dots,\widetilde\tau_m\}$ are forest-shape-Wilf equivalent.
\end{theorem}
As we will see, this result implies that the sets $\{\tau_1,\dots,\tau_m\}$ and $\{\widetilde\tau_1,\dots,\widetilde\tau_m\}$ are forest-structure-Wilf equivalent, so it also generalizes \cref{thm:westbijection}.

Finally, we study consecutive pattern avoidance in rooted forests. In the context of permutations, a \textit{consecutive instance} of a pattern $\sigma$ in a permutation $\pi$ is a consecutive subsequence of $\pi$ whose elements are in the same relative order as $\sigma$. If no such subsequence exists, then $\pi$ is said to \textit{avoid} $\sigma$ (as a consecutive pattern). The analogous notion of Wilf equivalence is known as \textit{c-Wilf equivalence}. Two patterns $\sigma$ and $\tau$ are called \textit{strong-c-Wilf equivalent} if for all $n$ and $m$, the number of length-$n$ permutations with exactly $m$ consecutive instances of $\sigma$ equals the number of length-$n$ permutations with exactly $m$ consecutive instances of $\tau$. Elizalde and Noy used the cluster method in \cite{elizaldenoy2}, introduced by Goulden and Jackson in 1979 \cite{gouldenjackson}, to study consecutive pattern avoidance in permutations and c-Wilf equivalence. Elizalde and Dwyer conjectured the following in \cite{dwyer}, which was proved by Lee and Sah in \cite{leesah}.

\begin{theorem}[{\cite[Corollary~1.2]{leesah}}]
If the patterns $\sigma,\tau\in\mathcal{S}_k$ are strong-c-Wilf equivalent, then $\{\sigma(1), \sigma(k)\}=\{\tau(1), \tau(k)\}$ or $\{k+1-\sigma(1), k+1-\sigma(k)\}=\{\tau(1), \tau(k)\}$.
\end{theorem}

We call the analogous notions of forest-Wilf equivalence for consecutive pattern avoidance \textit{c-forest-Wilf equivalence} and \textit{strong-c-forest-Wilf equivalence}, and prove the following statement constraining such equivalences, a result analogous to that obtained by Lee and Sah:

\begin{theorem}\label{thm:same-first-number}
If patterns $\sigma,\tau\in\mathcal{S}_k$ are strong-c-forest-Wilf equivalent, then $\sigma(1)=\tau(1)$ or $\sigma(1)+\tau(1)=k+1$.
\end{theorem}

Using the cluster method, we also prove the following surprising result, which gives the only nontrivial single-pattern c-forest-Wilf equivalence for pattern length at most $5$:

\begin{theorem}\label{thm:1324-1423}
The patterns $1324$ and $1423$ are strong-c-forest-Wilf equivalent, and therefore c-forest-Wilf equivalent.
\end{theorem}

The patterns $1324$ and $1423$ are not even c-Wilf equivalent (in the context of consecutive pattern avoidance in permutations) \cite[Table~3]{elizaldenoy1}, and therefore no bijection between $n$-vertex forests avoiding $1324$ and $n$-vertex forests avoiding $1423$ could be structure-preserving. This suggests that the study of consecutive pattern avoidance in forests is more subtle than a mere special case of the study in permutations.

We first give preliminary definitions in \cref{sec:preliminaries}. In \cref{sec:recurrences}, we provide recurrences for enumerating the forests avoiding the patterns $123\cdots k$, $213$, and for avoiding various other sets of patterns. The primary focus of \cref{sec:forest-wilf-equivalences} is to prove \cref{thm:westbijection}, generalizing the result by Anders and Archer that the patterns $123$ and $132$ are forest-Wilf equivalent to single patterns of any length. In \cref{sec:forest-shape-wilf-equivalences}, we then introduce the notion of forest-shape-Wilf equivalence to prove \cref{thm:fswe}, generalizing \cref{thm:westbijection} to sets of multiple patterns of arbitrary length. Finally, in \cref{sec:consecutive}, we discuss consecutive pattern avoidance in forests as defined in \cite{andersarcher}, and prove \cref{thm:same-first-number,thm:1324-1423}.

\section{Preliminaries}\label{sec:preliminaries}
Let $\mathcal{S}_n$ be the set of permutations on $[n]\coloneqq\{1, 2, \dots, n\}$. A permutation $\pi=\pi(1)\pi(2)\cdots\pi(n)\in\mathcal{S}_n$ \textit{contains} a pattern $\sigma=\sigma(1)\sigma(2)\cdots\sigma(k) \in\mathcal{S}_k$ if there is a sequence $1 \le a_1 < a_2 < \cdots < a_k \le n$ such that $\pi(a_1)\pi(a_2)\cdots\pi(a_k)$ is in the same relative order as $\sigma(1)\sigma(2)\cdots\sigma(k)$. We generally write a pattern as a permutation of positive integers from $1$ to $k$, where $k$ is the length of the pattern; for example, $213$ is a valid length-$3$ pattern. We can generalize this notion to arbitrary sequences of distinct positive integers: a sequence $\pi$ \textit{contains} a pattern $\sigma$ if it has a subsequence with integers in the same relative order as $\sigma$. Otherwise, $\pi$ \textit{avoids} $\sigma$.

In a rooted tree, we define the \textit{parent} of a non-root vertex $v$ to be the vertex directly preceding it in the shortest path from the root to $v$. Any non-root vertex is a \textit{child} of its parent. The \textit{ancestors} of a vertex are the vertices on the path from the root to the vertex, including itself, and a vertex is a \textit{descendant} of each of its ancestors. The set of all descendants of a vertex $v$, including $v$ itself, form the \textit{subtree rooted at $v$}, which can be considered as a rooted tree with root $v$. A vertex $v'$ is a \textit{strict ancestor} of a vertex $v$ if $v'$ is an ancestor of $v$, and $v'\neq v$. We define \textit{strict descendants} similarly. We define the \textit{depth} of a vertex in a rooted tree inductively; the depth of the root equals $1$, and each other vertex has depth equal to one greater than the depth of its parent. The \textit{depth} of a rooted tree is defined to be the maximum depth of its vertices.

We use the natural convention that the empty graph is a rooted forest, but not a rooted tree; in other words, rooted trees must contain at least one vertex.

The labels of a labeled forest or labeled tree are always assumed to be positive integers. Unless otherwise specified, the labels are also assumed to be distinct. The \textit{label set} of a rooted labeled forest $F$ (resp.\ tree $T$) is the set of labels of the vertices of $F$ (resp.\ $T$). Given a finite set $L$ of positive integers, by a \textit{forest on $L$} (resp. \textit{tree on $L$}), we mean a rooted labeled forest (resp.\ tree) with label set $L$; note this implies the forest (resp.\ tree) has $|L|$ vertices. In particular, a forest or tree on $[n]$ has exactly $n$ vertices, which are labeled $1$ through $n$.

A rooted labeled forest is \textit{increasing} if the label of each vertex is less than the labels of all its children, and \textit{decreasing} if the label of each vertex is greater than the labels of all its children.

\begin{figure}
\includegraphics{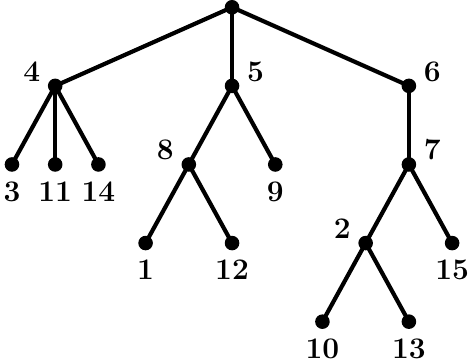}
\caption{An example of a forest on $[15]$. This forest consists of three trees, with roots labeled $4$, $5$, and $6$.}
\label{fig:forestexample}
\end{figure}

For convenience, when drawing rooted forests we often connect the root of each constituent rooted tree to an additional unlabeled vertex. However, this vertex is not part of the forest and only is there to aid in visualization. As in \cref{fig:forestexample}, we visualize each rooted tree as having its root at the top, and arranged so that each parent is placed above its children. So, the vertices strictly below a given vertex (and in its subtree) are the strict descendants of that vertex. Furthermore, in each rooted tree, the shortest path from the root to any of its descendants forms a downward path.

For a positive integer $k$, an \textit{instance} of a pattern $\sigma=\sigma(1)\cdots\sigma(k)\in\mathcal{S}_k$ in an unordered rooted labeled forest $F$ is a sequence $v_1,\dots,v_k$ of vertices of $F$ such that $v_i$ is an ancestor of $v_j$ for all $1\le i<j\le k$ and the labels of $v_1,\dots,v_k$ are in the same relative order as $\sigma(1)\cdots\sigma(k)$. We say that $v_1$ is the \textit{starting point} of this instance, and that $v_k$ is its \textit{endpoint}.

If there is at least one instance of $\sigma$ in $F$, we say $F$ \textit{contains} $\sigma$. Otherwise, we say $F$ \textit{avoids} $\sigma$. Note that a rooted tree avoids $\sigma$ if and only if along every downward path from the root to a vertex, the sequence of labels obtained avoids $\sigma$. Furthermore, a forest of rooted trees avoids $\sigma$ if and only if each of its trees avoids $\sigma$. For instance, the forest shown in \cref{fig:forestexample} avoids the pattern $132$. Given a set $S=\{\sigma_1,\dots,\sigma_m\}$ of patterns, we say a forest or tree \textit{avoids} $S$ if it avoids each of $\sigma_1,\dots,\sigma_m$. For convenience, we will often refer to the singleton set $\{\sigma\}$ as simply $\sigma$.

For a nonnegative integer $n$, let $F_n$ be the set of unordered rooted labeled forests on $[n]$, and let $T_n$ be the set of unordered rooted labeled trees on $[n]$. For a set of patterns $S$, let $F_n(S)$ be the set of forests in $F_n$ avoiding $S$, and let $f_n(S)=|F_n(S)|$. Similarly, let $T_n(S)$ be the set of trees in $T_n$ avoiding $S$, and let $t_n(S)=|T_n(S)|$. In particular, we have $f_0(S)=1$ and $t_0(S)=0$ for all $S$ (note $T_0$ is empty). Also, for notational convenience, we often write $F_n(\sigma_1,\dots,\sigma_m)$ to mean $F_n(\{\sigma_1,\dots,\sigma_m\})$, and similarly for $T_n$, $f_n$, and $t_n$.

Two sets of patterns $S$ and $S'$ are \textit{Wilf equivalent} if for all $n$, the number of permutations in $\mathcal{S}_n$ avoiding $S$ is the same as the number of permutations in $\mathcal{S}_n$ avoiding $S'$. We say $S$ and $S'$ are \textit{forest-Wilf equivalent} if for all $n$ we have $f_n(S)=f_n(S')$.

For a pattern $\sigma \in\mathcal{S}_k$, its \textit{complement} $\sigma^c\in\mathcal{S}_k$ is formed by defining $\sigma^c(i)=k+1-\sigma(i)$ for all $1\le i\le k$. If a forest on $[n]$ avoids $\sigma$, then by replacing each vertex $i$ with $n+1-i$, the resulting forest avoids $\sigma^c$. For a set of patterns $S=\{\sigma_1,\dots,\sigma_m\}$, let $S^c$ denote the corresponding set of complements $\{\sigma_1^c, \dots, \sigma_m^c\}$. Then $f_n(S)=f_n(S^c)$ for all $n$, so $S$ and $S^c$ are forest-Wilf equivalent. This fact is also Proposition 1 in \cite{andersarcher}.

For example, as briefly mentioned in the introduction, increasing forests are exactly those forests that avoid $21$; similarly, decreasing forests are exactly those forests that avoid $12$. Since the patterns $12$ and $21$ are complements, they are forest-Wilf equivalent. The expression $f_n(21)$ counts the number of increasing forests on $[n]$, and $f_n(12)$ counts the number of decreasing forests on $[n]$; it is well-known that $f_n(21)=f_n(12)=n!$, which can be proven using induction or with a bijection \cite{andersarcher}.

Two sets of patterns $S, S'$ are \textit{forest-structure-Wilf equivalent} if for any fixed rooted forest, the number of labelings of the forest that avoid $S$ equals the number of labelings that avoid $S'$. Note that this notion is stronger than forest-Wilf equivalence.

\section{Recurrences}\label{sec:recurrences}
In \cite{andersarcher}, Anders and Archer enumerated $f_n(S)$ for $S=\{213, 312\}, \{213, 312, 321\}$, $\{213, 312, 123\}$, $\{213, 312, 132\}, \{213, 132, 321\}, \{321, 2143, 3142\},$ and their complements. They also provided a recurrence to calculate $f_n(\{213, 312, 321\})$. In this section we detail how to find recurrences for counting forests avoiding certain sets of patterns. We can always count both the number of rooted trees and the number of forests avoiding a set of patterns, and we get recurrences between the two. For convenience, in this section we identify a vertex with its label. For example, for vertices $v$ and $w$, we say $v<w$ if $v$ has a smaller label than $w$ does.

\subsection{Forests from trees}
\label{sec:treeforest}
There is a recurrence for deriving the number of forests avoiding a set of patterns $S$ from the number of trees avoiding $S$ that does not depend on $S$. This recurrence was implicitly derived by Anders and Archer in the proof of Theorem 14 from \cite{andersarcher}, where they determine a recurrence for forests avoiding $\{213, 312, 231\}$.

Consider a property on trees and forests such that the following are true:
\begin{itemize}
    \item The property holds for a forest $F$ if and only if it holds for each constituent tree of $F$.
    \item If the property holds for a tree $T$, it holds for any tree $T'$ with the same underlying tree structure as $T$ and vertices labeled in the same relative order.
\end{itemize}
For example, the property can be avoiding a set of patterns $S$. We use the notation $T(n)$ (resp.\ $F(n)$) to denote the number of rooted trees (resp.\ forests) on $[n]$ satisfying this condition, where as usual $T(0)=0$ and $F(0)=1$. Informally, we call such a forest (resp.\ tree) valid. Consider a valid forest on $[n]$. If vertex $1$ is in a tree with $i-1$ other vertices, there are $\binom{n-1}{i-1}$ ways to choose these vertices, $T(i)$ ways to form a valid rooted tree on these vertices, and $F(n-i)$ ways to form a valid forest from the remaining vertices. So by summing over all possible $i$ (the number of vertices in the tree with vertex $1$), we arrive at the recurrence

\begin{equation}\label{eq:forest-from-tree}
    F(n)= \sum_{i=1}^{n} \binom{n-1}{i-1}F(n-i)T(i).
\end{equation}

Additionally, let $F(n, m)$ be the number of valid forests on $[n]$ with $m$ possible (distinguishable) pots to put the trees in, so $F(n)=F(n, 1)$. More precisely, we mean that $F(n, m)$ is the number of valid forests $F$ on $[n]$ where we assign each constituent tree of $F$ to exactly one of $m$ distinguishable pots, some of which may be empty. A given tree containing vertex $1$ can go into one of $m$ pots, so our recurrence includes a factor of $m$. Thus, similarly to the above, we obtain

\begin{equation}\label{eq:forest-from-tree-2}
F(n, m)= m \sum_{i=1}^{n} \binom{n-1}{i-1}F(n-i, m)T(i).
\end{equation}

Note that this equation directly generalizes \cref{eq:forest-from-tree}, which is the case $m=1$.

Let $f_m=\sum_{n=0}^{\infty} F(n, m)x^n/n!$ and $t=\sum_{n=0}^{\infty} T(n)x^n/n!$ be the exponential generating functions for $F(n, m)$ and $T(n)$, respectively. The recurrence for $F(n, m)$ can be written as a convolution, giving $f_m'=mf_mt'$, where $f_m', t'$ are the derivatives of $f_m$ and $t$, respectively, with respect to $x$. Therefore, we have $t=\log(f)/m$.

\subsection{Forests avoiding sets containing $213$}
In this section, we give recurrences for the number of forests avoiding sets of patterns $S$ such that $213 \in S$. We already have a recurrence for the number of forests avoiding $S$ in terms of the number of trees avoiding $S$, so in each section we find a recurrence for the number of trees avoiding $S$. In each scenario, $T(n)$, $F(n)$, and $F(n, m)$ are defined as in \cref{sec:treeforest} for a given $S$. Each recurrence involves building a forest or a tree from smaller forests and trees, and in each case it will be clear that the ``decomposition'' can be reversed uniquely.

\subsubsection{Forests avoiding $\{213, 231\}$}
Let $S=\{213, 231\}$. Consider a rooted tree on $[n]$ avoiding $S$ with root $i$. For vertices $a, b$, if $a>i$ and $b<i$, then no downward path from $i$ can contain both $a$ and $b$, because otherwise the tree contains either $213$ or $231$. Therefore, every subtree of $i$ (i.e., every subtree rooted at a child of $i$) contains vertices that either are all greater than $i$ or are all less than $i$. Ignoring the root, such a tree is simply a forest on $\{i+1, \dots, n\}$ avoiding $S$ combined with a forest on $\{1, 2, \dots, i-1\}$ avoiding $S$, so we arrive at the recurrence

\[T(n)=\sum_{i=1}^n F(i-1)F(n-i).\]

\subsubsection{Forests avoiding $\{213\}$}\label{subsubsec:213-recurrence}
Let $S=\{213\}$. Consider a rooted tree on $[n]$ avoiding $S$ with root $i$. For vertices $a < i < b$, as the tree avoids $213$, $a$ cannot be an ancestor of $b$. So, any downward path from $i$ to a vertex greater than $i$ only contains $i$ and vertices greater than $i$, meaning that the set of vertices $\{i, i+1, \dots, n\}$ forms a contiguous tree, i.e., the subgraph induced by $\{i, i+1, \dots, n\}$ is a tree rooted at $i$. Ignoring the root, this contiguous tree is a forest avoiding $S$ on the set of vertices $\{i+1, \dots, n\}$ The number of ways to form this forest is $F(n-i)$.

Given the contiguous tree on $\{i, i+1, \dots, n\}$, we need to add the vertices $1, 2, \dots, i-1$. No instance of a $213$ pattern can be formed that includes a vertex in the contiguous tree. Then, the structure made by the set of vertices $\{1, 2, \dots, i-1\}$ is a forest on $[i-1]$ avoiding $S$ with $n-i+1$ distinguishable pots to put the trees in, depending on what root vertex the tree is attached to. Specifically, these possible root vertices are the vertices $i, i+1, \dots, n$. Therefore, we arrive at the recurrence

\[T(n)=\sum_{i=1}^n F(n-i)F(i-1, n-i+1).\]

\subsubsection{Forests avoiding $\{213, 123\}$ or $\{213, 132\}$}\label{subsubsec:213-123-recurrence}
Suppose $S=\{213, 123\}$ or $S=\{213, 132\}$. Consider a rooted tree on $[n]$ avoiding $S$ with root $i$. As in \cref{subsubsec:213-recurrence}, the vertices $\{i+1, \dots, n\}$ form a contiguous tree with root $i$. If $S=\{213, 123\}$, then these vertices form a decreasing forest on $n-i$ vertices, i.e., every vertex that is not a child of $i$ is smaller than its parent. If $S=\{213, 132\}$, then these vertices form an increasing forest on $n-i$ vertices. Either way, there are $(n-i)!$ ways to create such a forest, as discussed in \cref{sec:preliminaries}.

Given the contiguous tree on $\{i, i+1, \dots, n\}$, we need to add the remaining vertices $1, 2, \dots, i-1$. In either case, no pattern in $S$ can include a vertex from the contiguous tree. Therefore, as in the previous section, we arrive at the recurrence
\[T(n)=\sum_{i=1}^n (n-i)!F(i-1, n-i+1).\]

\subsubsection{Forests avoiding $\{213, 321\}$}
Let $S=\{213, 321\}$. Further, define $F(n, m, r)$ to be the number of forests avoiding $S$ on $[n]$ with $m$ pots to put the trees in, such that exactly $r$ vertices are not the endpoint of an instance of the pattern $21$. Define $T(n, r)$ to be the number of rooted trees avoiding $S$ on $[n]$ with exactly $r$ vertices that are not a descendant of the endpoint of an instance of $21$. Note that this includes not being the endpoint of an instance of $21$, since every vertex is a descendant of itself. We have a modified recurrence for $F$ now: as in Equation \cref{eq:forest-from-tree-2}, we sum over $i$ (the number of vertices in the tree with vertex $1$), but now after splitting into a tree of size $i$ and a forest of size $n-i$, we sum over $\ell$, the number of vertices in the size-$i$ tree that are not a descendant of the endpoint of an instance of $21$. We arrive at \[F(n, m, r)= m\sum_{i=1}^{n} \binom{n-1}{i-1} \sum_{\ell=0}^i T(i, \ell)F(n-i, m, r-\ell).\]

We now derive a recurrence for $T(n, r)$. Consider a rooted tree on $[n]$ avoiding $S$ with root $i$. As in \cref{subsubsec:213-recurrence}, the vertices $i, i+1, \dots, n$ form a contiguous tree with root $i$. Note that $i$ cannot contribute to any instance of $21$ in this tree. So, the vertices $i+1, \dots, n$ form a forest on $n-i$ vertices avoiding $S$.

Given a vertex $v$ in our contiguous tree that is a descendant of the endpoint of an instance of $21$, $v$ cannot have any children in $\{1, 2, \dots, i-1\}$, as that child would be the endpoint of an instance of $321$. Then, if there are $\ell$ vertices in our contiguous tree not satisfying this criterion, we can attach each of $1, 2, \dots, i-1$ to any of those $\ell$ vertices. Note that every vertex in $\{1, 2, \dots, i-1\}$, being a descendant $i$, is an endpoint of an instance of $21$, and we end up with $\ell$ vertices that are not a descendant of the endpoint of an instance of $21$. To find $T(n, r)$, when considering trees with root $i$, we must take $\ell=r$. Since $i$ itself is not a descendant of the endpoint of an instance of $21$, there must be $r-1$ such vertices in $\{i+1, \dots, n\}$, meaning that there are $F(n-i, 1, r-1)$ ways to form a forest on those vertices. The forest on $\{1, 2, \dots, i-1\}$ is a forest with $r$ pots to put the trees in, and it must be increasing, as if $21$ occurs in the forest on $[i-1]$, then adding $i$ to the instance of $21$ creates an instance of $321$ with starting point $i$. The number of ways to create an increasing forest on $[i-1]$ with $r$ pots to put the trees in is $r(r+1)\cdots(r+i-2)=(r+i-2)!/(r-1)!$, since there are $r$ places to put $1$, then $r+1$ places to put $2$, and so on. Thus, we arrive at the recurrence

\[ T(n, r) = \sum_{i=1}^{n} F(n-i, 1, r-1)\frac{(r+i-2)!}{(r-1)!}.\]

\subsubsection{Forests avoiding $\{213, 123, 132\}$}
Let $S=\{213, 123, 132\}$. Consider a rooted tree on $[n]$ avoiding $S$ with root $i$. For vertices $a, b > i$, no downward path from $i$ can contain both $a$ and $b$, since the tree avoids $123$ and $132$. So, since the tree also avoids $213$, the vertices $i+1, i+2, \dots, n$ must all be children of $i$. The structure on the remaining vertices $1, 2, \dots, i-1$ is a forest with $n-i+1$ distinguishable pots to put the trees in. So, we arrive at the recurrence

\[T(n) = \sum_{i=1}^n F(i-1, n-i+1).\]

\subsubsection{Forests avoiding $\{213, 231, 123\}$ or $\{213, 231, 132\}$}
Let $S=\{213, 231, 123\}$ or $\{213, 231, 132\}$. Note that for the latter set of patterns, the recurrence below is already given in \cite{andersarcher}. Consider a rooted tree on $[n]$ avoiding $S$ with root $i$. As in the case of avoiding $\{213, 231\}$, in a rooted tree avoiding $S$ with root $i$, the vertices $i+1, i+2, \dots, n$ form a forest, as do the vertices $1, 2, \dots, i-1$. In addition, we have that the forest on $\{i+1, \dots, n\}$ is increasing if $S=\{213, 231, 123\}$, and decreasing if $S=\{213, 231, 132\}$. The forest on $\{1, 2, \dots, i-1\}$ can be any forest on $[i-1]$ avoiding $S$. So, we arrive at the recurrence

\[T(n)=\sum_{i=1}^n F(i-1)(n-i)!.\]

\subsection{Forests avoiding $k(k-1)(k-2)\cdots1$}
We give a $(k-1)$-dimensional recurrence, i.e., the recurrence has $k-1$ parameters. Let \linebreak $F(n, a_1, a_2, \dots, a_{k-1})$ (resp.\ $T(n, a_1, a_2, \dots, a_{k-1})$) be the number of forests (resp.\ trees) on $[n]$ with the following property: for $2 \le j \le k$, each instance of the pattern $j(j-1)(j-2)\cdots1$ that appears in the forest has starting point strictly greater than $a_{j-1}$. For example, $F(n, n-1)$ is the set of all forests on $[n]$ such that no instance of $21$ appears with starting point at most $n-1$; these are forests that are increasing except that vertex $n$ can go anywhere. Then, the number of forests avoiding $j(j-1)(j-2)\cdots1$ is $F(n, \underbrace{0, 0, \dots, 0}_{j-2 \text{ 0's}}, n)$.

If $a_\ell \ge a_{\ell+1}$, then clearly
\[F(n, a_1, \dots, a_\ell, a_{\ell+1}, \dots, a_{k-1})=F(n, a_1, \dots, a_\ell, a_\ell+1, a_{\ell+2}, a_{\ell+3}, \dots, a_{k-1}), \]
as there being no instances of $(\ell+1)\ell\cdots1$  with a starting point at most $a_\ell$ implies that there are no instances of $(\ell+2)(\ell+1)\ell\cdots1$ with a starting point at most $a_\ell+1$. So, we can always replace $a_\ell$ with $\max(a_\ell, a_{\ell-1}+1)$ in $F(n, a_1, a_2, \dots, a_{k-1})$, and similarly in $T(n, a_1, a_2, \dots, a_{k-1})$. Unless otherwise specified, the parameters are assumed to be of this form in the recurrences below. So, $F(n, \underbrace{0, 0, \dots, 0}_{j-2 \text{ 0's}}, n)$ is equal to $F(n, 1, 2, \dots, j-2, n)$, with both values equal to the number of forests on $[n]$ avoiding the pattern $j(j-1)\cdots1$.

Now we determine a recurrence for $F$ and $T$. For $T$, we do casework based on what the root vertex is. Consider a tree counted by $T(n, a_1, \dots, a_{k-1})$, such that for all $2 \le \ell \le k-1$, either $a_\ell \ge a_{\ell-1}+1$ or $a_\ell=n$. Let $i$ be the root vertex.\\

\noindent\textbf{Case 1}: $i=1$. The remaining part of the tree without $i$ is a forest on $[n-1]$ with each label incremented by $1$, so there are $F(n-1, a_1-1, a_2-1, \dots, a_{k-1}-1)$ such trees. \\
\\
\textbf{Case 2}: $2 \le i \le a_1$. If the root vertex is between $2$ and $a_1$, then since the vertex $1$ must be somewhere below, there is an instance of $21$ with starting point at most $a_1$, which is not possible. So, there are $0$ such trees. \\
\\
\textbf{Case 3}: $a_\ell+1 \le i \le a_{\ell+1}$. We have $n-1$ vertices left. Consider the vertices to be relatively labeled as $1, 2, \dots, n-1$, so everything less than $i$ stays the same and everything greater than $i$ is reduced by $1$. For a pattern $j(j-1)(j-2)\cdots1$ where $2 \le j \le \ell$, the same restriction applies: no instance of the pattern appears if the starting point is at most $a_{j-1}$. However, an instance of the pattern $(\ell+1)\ell\cdots1$ with starting point less than $i$ will create an instance of $(\ell+2)(\ell+1)\cdots1$ with starting point equal to $i \le a_{\ell+1}$ (by appending the root vertex $i$), which is not possible. \\
\\
So, all instances of $(\ell+1)\ell\cdots1$ must have starting point greater than $i$, which then, in the forest created by removing the root $i$ and relabeling the remaining vertices on $[n-1]$, must have starting point at least $i$. So, we set $a_\ell$ to $i-1$. Finally, for $j \ge \ell+2$, all instances of patterns of the form $j(j-1)\cdots 1$ also have starting point greater than $i$, and therefore have a starting point with a reduced label, so we reduce the requirement by $1$. Thus, the number of valid tree constructions with a root of $i$ is $F(n-1, a_1, a_2, \dots, a_{\ell-1}, i-1, a_{\ell+1}-1, a_{\ell+2}-1, \dots, a_{k-1}-1)$.

Putting it all together, we have
\begin{align*}
    T(n, a_1, \dots, a_{k-1})&=F(n-1, a_1-1, \dots, a_{k-1}-1)\\
    &\qquad+\sum_{i=a_1+1}^{a_2}F(n-1, i-1, a_2-1, a_3-1, \dots, a_{k-1}-1)\\
    &\qquad+\sum_{i=a_2+1}^{a_3} F(n-1, a_1, i-1, a_3-1, a_4-1, \dots, a_{k-1}-1)\\
    &\qquad+\sum_{i=a_3+1}^{a_4} F(n-1, a_1, a_2, i-1, a_4-1, a_5-1, \dots, a_{k-1}-1)\\
    &\qquad\vdotswithin{+}\\
    &\qquad+\sum_{i=a_{k-1}+1}^{a_k} F(n-1, a_1, a_2, \dots, i-1, a_{k-1}-1)\\
    &\qquad+\sum_{i=a_k+1}^{n} F(n-1, a_1, a_2, \dots, a_{k-2}, i-1).
\end{align*}

Now, we find a recurrence for $F$. The idea is essentially the same as in previous recurrences: we do casework based on the set of vertices that are in the same tree as vertex $1$. For this recurrence, we care a little about where the vertices come from (more than just the total number of vertices in the same tree as $1$). We split up ${2, \dots, n}$ into regions $[2, a_1], [a_1+1, a_2], [a_2+1, a_3], \dots, [a_{k-2}+1, a_{k-1}], [a_{k-1}+1, n]$, and do casework by the number of vertices we take from each region. Without loss of generality, set $a_1$ to be at least $1$, since the $a_1=1$ and $a_1=0$ cases are the same.

For $2 \le \ell \le k-1$, let $b_\ell$ be the number of vertices taken from the region $[a_{\ell-1}+1, a_\ell]$. Also, let $b_1, b_k$ be the number of vertices taken from $[2, a_1]$ and $[a_{k-1}+1, n]$, respectively. The number of possible trees including $1$ with these vertices is then
\[T\left(1+\sum_{\ell=1}^{k} b_\ell, 1+b_1, 1+b_1+b_2, \dots, 1+b_1+b_2+\cdots+b_{k-1}\right).\]
This is because if we label the vertices by their relative order $1, 2, \dots, 1+\sum_{\ell=1}^{k} b_\ell$, then any instance of $j(j-1)\cdots1$ in the first $1+b_1+\cdots+b_{j-1}$ vertices gives an instance of $j(j-1)\cdots1$ with starting point at most $a_{j-1}$ in the original forest, which is not possible. Similarly, on the remaining vertices, the number of possible forests is \[F\left(n-1-\sum_{\ell=1}^{k} b_\ell, a_1-1-b_1, a_2-1-b_1-b_2, a_3-1-b_1-b_2-b_3, \dots, a_{k-1}-1-\sum_{\ell=1}^{k-1} b_\ell\right).\]

Thus, we arrive at the following recurrence:

\begin{align*}
F(n, a_1, a_2, \dots, a_{k-1}) &= \sum_{b_1=0}^{a_1-1} \sum_{b_2=0}^{a_2-a_1} \cdots \sum_{b_{k-1}=0}^{a_{k-1}-a_{k-2}}\sum_{b_{k}=0}^{n-a_{k-1}}\Bigg[\binom{a_1-1}{b_1} \binom{a_2-a_1}{b_2}\cdots\binom{a_{k-1}-a_{k-2}}{b_{k-1}}\binom{n-a_{k-1}}{b_{k}}\\
&\qquad\times T\left(1+\sum_{\ell=1}^{k} b_\ell, 1+b_1, 1+b_1+b_2, \dots, 1+b_1+b_2+\cdots+b_{k-1}\right)\\
&\qquad\times F\left(n-1-\sum_{\ell=1}^{k} b_\ell, a_1-1-b_1, a_2-1-b_1-b_2, \dots, a_{k-1}-1-\sum_{\ell=1}^{k-1} b_\ell\right)\Bigg].
\end{align*}

The number of forests on $[n]$ avoiding the pattern $k(k-1)\cdots1$ is equal to $F(n, \underbrace{0, 0, \dots, 0}_{k-2 \text{ 0's}}, n)$. It turns out that we can reduce the dimension of the recurrence by $1$ in this specific case, where we avoid $k(k-1)\cdots1$ with no additional restrictions. We claim that when solving for this value using the recurrence, the first and last parameters are always the same. If we start with $F(i, \underbrace{0, 0, \dots, 0}_{k-2 \text{ 0's}}, i)$ for some $i$, then in the recurrence for $F$, we have that $b_{k}$ as defined above is always $0$, meaning that in \[T\left(1+\sum_{\ell=1}^{k} b_\ell, 1+b_1, 1+b_1+b_2, \dots, 1+b_1+b_2+\cdots+b_{k-1}\right),\] the first and last parameters are the same. In the recurrence for $F$, the term \[F\left(n-1-\sum_{\ell=1}^{k} b_\ell, a_1-1-b_1, a_2-1-b_1-b_2, a_3-1-b_1-b_2-b_3, \dots, a_{k-1}-1-\sum_{\ell=1}^{k-1} b_\ell\right)\] also has equal first and last parameters. Finally, in the recurrence for any $F(i, \underbrace{0, 0, \dots, 0}_{k-2 \text{ 0's}}, i)$, all the terms are $F$ terms with first and last parameter $i-1$. 

Thus, we write new recurrences defining $F'(n, a_1, \dots, a_\ell)$ to be $F(n, a_1, \dots, a_\ell, n)$ and $T'(n, a_1, \dots, a_\ell)$ to be $T(n, a_1, \dots, a_\ell, n)$, with the number of forests avoiding $k(k-1)\cdots1$ to be $F'(n, \underbrace{0, 0, \dots, 0}_{k-2 \text{ 0's}})$. For example, we get a $2$-dimensional recurrence for avoiding $321$, where $F(n)=F'(n, 1)$ and $T(n)=T'(n, 1)$:
\begin{gather*}
    F'(n, m)=\sum_{i=0}^{m-1}\sum_{j=0}^{n-m} \binom{m-1}{i}\binom{n-m}{j}T'(1+i+j, i+1)F'(n-i-j-1, m-1-i),\\
    T'(n, m)=F'(n-1, m-1)+\sum_{i=m+1}^n F'(n-1, i-1).\\
\end{gather*}

\subsection{Forests avoiding $\{12,k(k-1)(k-2)\cdots 1\}$}
For convenience, we assume $k\ge 2$. Note that any nontrivial set of patterns containing $12$ is equivalent to $\{12\}$ or $\{12,k(k-1)\cdots 1\}$ for some $k\ge 2$.

The main observation is that a forest avoids both $12$ and $k(k-1)\cdots 1$ if and only if it is decreasing and the depth of each constituent rooted tree is at most $k-1$; such forests are discussed by Luschny in \cite{luschnytrees}.

As $k$ varies, the numbers of forests avoiding the sets of the form $\{12,k(k-1)\cdots 1\}$ are enumerated by sequences known as the higher-order Bell numbers, which were defined and studied by Luschny in \cite{luschnybell}. Following his work, we first define the Bell transform and the $i$th-order Bell numbers. The \textit{Bell transform} (as given in the ``Bell matrix" section of \cite{luschnybell}) takes a sequence $a_0,a_1,a_2,\dots$ and outputs a triangular array $\Delta$ with entries $\Delta(n,m)$, where $m,n$ are integers such that $0\le m\le n$; we define $\Delta(0,0)=1$, $\Delta(n,0)=0$ and $\Delta(n,1)=a_{n-1}$ for $n\ge 1$, and 
\[\Delta(n,m)=\sum_{j=1}^{n-m+1} \binom{n-1}{j-1}\Delta(n-j, m-1)\Delta(j,1)\]
for $2\le m\le n$.

Let $S_0$ be the sequence defined by $a_n=1$ for all $n\ge 0$. Now for all integers $i\ge 0$, let $\Delta_i(n,m)$ be the Bell transform of $S_i$. We also let $S_{i+1}$ be the sequence of row sums of the triangle $\Delta_i(n,m)$, so $(S_{i+1})_n=\sum_{m=0}^n \Delta_i(n,m)$. For all nonnegative integers $i$, we define the sequence of \textit{$i$th-order Bell numbers} to be $S_i$. The following result is already observed in sequence A179455 of the OEIS \cite{oeis}; we provide a proof.

\begin{proposition}
Let $k\ge 2$ be a positive integer. Then for all $n\ge 0$,
\[f_n(12,k(k-1)\cdots 1)=(S_{k-2})_n.\]
\end{proposition}

\begin{proof}
We induct on $k$, where the base case $k=2$ follows from the fact that $f_n(12,21)=1$ for all $n\ge 0$. Now suppose that $k\ge 3$, and $f_n(12,(k-1)(k-2)\cdots 1)=(S_{k-3})_n$ for all $n\ge 0$. To show that $f_n(12,k(k-1)\cdots 1)=(S_{k-2})_n$, it suffices to show that for $0\le m\le n$, the quantity $\Delta_{k-3}(n,m)$ equals the number of decreasing forests on $[n]$ consisting of exactly $m$ rooted trees such that the depth of each rooted tree is at most $k-1$. To do this, we induct on $n$. The $n=0$ case is trivial, so let $n\ge 1$. Within the inductive step on $n$, we induct on $m$. The cases $m=0,1$ are easy, so we consider the case $2\le m\le n$. Vertex $n$ must be a root of any decreasing forest on $[n]$. Considering the rooted tree containing $n$ and using the inductive hypotheses, we find that the number of decreasing forests on $[n]$ consisting of exactly $m$ rooted trees each with depth at most $k-1$ equals
\[\sum_{j=1}^{n-m+1}\binom{n-1}{j-1}f_{j-1}(12,(k-1)(k-2)\cdots 1)\Delta_{k-3}(n-j,m-1).\]
By the outermost induction, we know $f_{j-1}(12,(k-1)(k-2)\cdots 1)=(S_{k-3})_{j-1}=\Delta_{k-3}(j,1)$. Thus
\[\sum_{j=1}^{n-m+1}\binom{n-1}{j-1}f_{j-1}(12,(k-1)(k-2)\cdots 1)\Delta_{k-3}(n-j,m-1)=\Delta_{k-3}(n,m),\]
so we are done.
\end{proof}

\section{Forest-Wilf equivalences}\label{sec:forest-wilf-equivalences}
In \cite{andersarcher}, Anders and Archer proved that $123$ and $132$ are forest-Wilf equivalent. Hopkins and Weiler implied the same result as a corollary of Theorem 3 from \cite{poset}, a more general result about pattern avoidance in posets. Both proofs of this fact actually imply that $123$ and $132$ are forest-structure-Wilf equivalent. Here we prove \cref{thm:westbijection} in \cref{subsec:west-bijection-proof}, generalizing this result. In \cref{subsec:restrict-west-bijection}, we then restrict the bijection to find families of inequalities and forest-Wilf equivalences between pairs of patterns, proving \cref{prop:forest-wilf-inequality,thm:213-equivalences}.

\subsection{Generalizing the forest-structure-Wilf equivalence of $123$ and $132$}\label{subsec:west-bijection-proof}
In this section we prove \cref{thm:westbijection}. Using the notation in the statement of \cref{thm:westbijection}, let $k\ge 3$ be an integer, and choose $\tau\in\mathcal{S}_k$ such that $\tau(k-1)=k-1$ and $\tau(k)=k$, and define $\widetilde\tau\in\mathcal{S}_k$ by letting $\widetilde\tau(i)=\tau(i)$ for $1\le i\le k-2$, $\widetilde\tau(k-1)=k$, and $\widetilde\tau(k)=k-1$. We wish to show that $\tau$ and $\widetilde\tau$ are forest-structure-Wilf equivalent. Note that this includes the pair $\tau=123,\widetilde\tau=132$ as a special case.

In \cite{westthesis}, West gave a generalization of the Simion--Schmidt bijection from \cite{simionschmidt} to show that $\tau$ and $\widetilde\tau$ are Wilf equivalent permutation patterns. Our proof of \cref{thm:westbijection} combines his methods with those used in \cite{andersarcher}. First, let $\bar\tau\in\mathcal{S}_{k-1}$ be the pattern defined by letting $\bar\tau(i)=\tau(i)=\widetilde\tau(i)$ for $1\le i\le k-2$ and $\bar\tau(k-1)=k-1$.
\begin{definition}
A vertex $v$ of a forest $F\in F_n$ is \textit{special} if there exists an instance of $\bar\tau$ that has $v$ as its endpoint. We say this instance of $\bar\tau$ \textit{establishes} that $v$ is special.
\end{definition}

\begin{figure}
\includegraphics{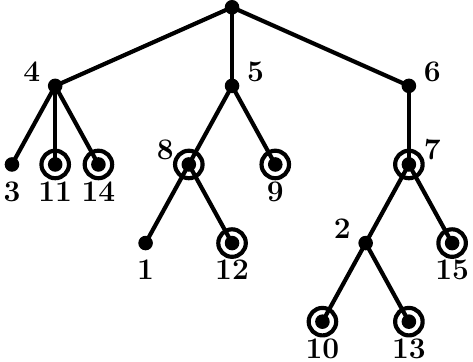}
\caption{The example forest from \cref{fig:forestexample}, which lies in $F_n(\widetilde\tau)$, where $\tau=123$. Its special vertices (with respect to $\tau=123$) are circled. For example, the vertices labeled $6$ and $10$ establish that the vertex labeled $10$ is special.}
\label{fig:forest1withcircles}
\end{figure}

For an example of special vertices, see \cref{fig:forest1withcircles}. We now define two operations on a special vertex $v$. We view each operation as fixing the vertices of $F$, but permuting the labels among the subtree rooted at $v$.

\begin{definition}
Let $v$ be a special vertex of a forest, and let $L$ be the label set of the subtree rooted at $v$. Let $x$ be the largest label in $L$, and let $y$ be the smallest label in $L$ such that if $v$ were labeled with $y$, then the vertex $v$ would remain special.

To \textit{shuffle} vertex $v$, we first label $v$ with $x$, and then relabel the strict descendants of $v$ with the elements of $L\setminus\{x\}$ so that the initial relative order of the labels of the vertices is preserved.

To \textit{antishuffle} vertex $v$, we first label $v$ with $y$, and then relabel the strict descendants of $v$ with the elements of $L\setminus\{y\}$ so that the initial relative order of the labels of the vertices is preserved.
\end{definition}
Note that $x$ and $y$ depend only on $L$, and not on how the strict descendants of $v$ are labeled with $L$. Also, $y\in L$ exists since $v$ is assumed to be special. We note the following:

\begin{lemma}\label{lem:shuffle}
Applying the shuffle operation to a special vertex $v$ preserves the set of special vertices of $F$. Moreover, any vertex $u\in F$ that is not special retains its original label.
\end{lemma}

\begin{proof}
For the second statement, it suffices to only consider vertices $u$ in the subtree rooted at $v$. If the label of $u$ is initially smaller than the label of $v$, then $u$ retains its label after the shuffle. Thus, only vertices $u$ with labels at least as large as the label of $v$ can have their label change. Furthermore, if the label of $u$ is at least the label of $v$, then $u$ is special. These two observations imply the second statement.

For the first statement, let $u\in F$. We can again assume that $u$ lies in the subtree rooted at $v$. Suppose $u$ is special prior to the shuffle. If the label of $u$ is at least the label of $v$, then the new label of $u$ will also be at least the original label of $v$. Considering a sequence of vertices that established that $v$ was special, each member of which has its label preserved under the shuffle, one sees that $u$ will still be special after the shuffle. Now suppose the initial label of $u$ was less than the initial label of $v$. Then, the label of $u$ was the largest label among the labels of the sequence of vertices establishing $u$ was special. Each of these vertices has its label preserved, so $u$ remains special.

Now suppose $u$ is in the subtree rooted at $v$, and $u$ is special after the shuffle, but not special before the shuffle. Then before the shuffle, its label was less than the label of $v$; thus the shuffle preserved the label of $u$. Consider a sequence $u_1,\dots,u_{k-2},u$ establishing that $u$ is special after the shuffle. At least one of these vertices, say $u_i$, must have had its label change. Then, $u_i$ is a descendant of $v$, and its new label is at least the original label of $v$, and is thus larger than the new label of $u$. However, this contradicts the assumption on the sequence $u_1,\dots,u_{k-2},u$.
\end{proof}

We have an analogous lemma for antishuffles:

\begin{lemma}\label{lem:antishuffle}
Applying the antishuffle to a special vertex $v$ preserves the set of special vertices of $F$. Moreover, any vertex $u\in F$ that is not special retains its original label.
\end{lemma}

\begin{proof}
Similarly to the proof of \cref{lem:shuffle}, the second statement only entails checking vertices $u$ in the subtree rooted at $v$, and one can show directly that special vertices remain special.

Our proof of the remaining half of the first statement is slightly different from the analogous part of our proof of \cref{lem:shuffle}. Suppose $u$ is not special before the antishuffle, but is special after the antishuffle. As noted above, the label of $u$ is preserved. Suppose the sequence $u_1,\dots,u_{k-2},u$ of vertices establishes that $u$ is special after the antishuffle, so for some $1\le i\le k-2$, the label of $u_i$ changed during the antishuffle. Then, $u_i$ is a strict ancestor of $u$, and it must have been special before the antishuffle, so also special after the antishuffle. We then consider the sequence of vertices $u_1',\dots,u_{k-2}',u_i$ establishing that $u_i$ is special after the antishuffle. If the labels of each of $u_1',\dots,u_{k-2}'$ is preserved, then the sequence $u_1',\dots,u_{k-2}',u$ establishes that $u$ is special before the antishuffle, a contradiction. Thus, at least one of the vertices among $u_1',\dots,u_{k-2}'$ had its label change. We may iterate this process indefinitely, but since the forest is finite, this is a contradiction. Therefore, if $u$ is special after the antishuffle, then it must have been special before the antishuffle as well.
\end{proof}

Fix $n\ge 1$. We now define maps $\alpha,\beta\colon F_n\to F_n$, which we view as fixing the vertices of $F$ but permuting its labels (as was the case with shuffles and antishuffles).
\begin{definition}
The map $\alpha\colon F_n\to F_n$ is defined as follows: given $F\in F_n$, we perform a breadth-first search on the vertices of $F$, in reverse order, so that we end with the roots of the trees comprising $F$. If the vertex under consideration is special, then shuffle that vertex; otherwise, continue. The resulting labeled forest is $\alpha(F)$.

Similarly, the map $\beta\colon F_n\to F_n$ is defined as follows: given $F\in F_n$, we perform a breadth-first search on the vertices of $F$, in the usual order; at each vertex $v$, we antishuffle $v$ if $v$ is special, and otherwise continue. The resulting forest is $\beta(F)$.
\end{definition}
We see that $\alpha$ is well-defined, since by \cref{lem:shuffle}, the set of special vertices remains constant throughout the process, and moreover, the resulting labeled forest does not depend on the order in which the vertices are considered, as long as each vertex is considered after each of its strict descendants. Similarly, \cref{lem:antishuffle} demonstrates that $\beta(F)$ is well-defined. Note that as unlabeled rooted forests, $\alpha(F)$ and $\beta(F)$ are isomorphic to $F$ (i.e., they have the same structure).

\begin{figure}
\includegraphics{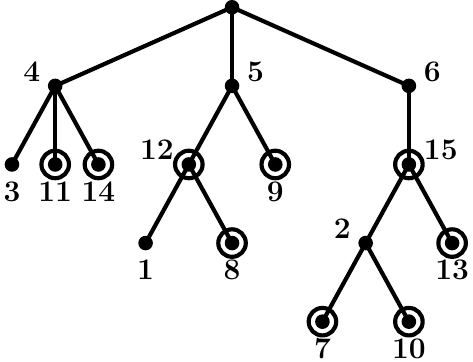}
\caption{The result of applying $\alpha$ to the forest in \cref{fig:forest1withcircles}, with respect to $\tau=123$. The special vertices are circled.}
\label{fig:forest2withcircles}
\end{figure}

As an example, the forest shown in \cref{fig:forest2withcircles} is the result of applying $\alpha$ to the forest shown in \cref{fig:forest1withcircles}. Note that the two forests have the same set of special vertices, and that the second forest avoids $123$.

\begin{lemma}
For all $F\in F_n$, we have $\alpha(F)\in F_n(\tau)$ and $\beta(F)\in F_n(\widetilde\tau)$.
\end{lemma}

\begin{proof}
Suppose first that $\alpha(F)$ contains $\tau$. Let $v_1,v_2,\dots,v_k$ be an instance of $\tau$ in $\alpha(F)$, so that $v_i$ is an ancestor of $v_j$ for $i<j$. Then, $v_{k-1},v_k$ are special in $\alpha(F)$, and the label of $v_{k-1}$ is less than the label of $v_k$. But considering the definition of $\alpha$, this is a contradiction. The proof that $\beta(F)\in F_n(\widetilde\tau)$ is similar.
\end{proof}

To prove \cref{thm:westbijection}, we will show that $\alpha$ and $\beta$ restrict to inverse maps $F_n(\widetilde\tau)\to F_n(\tau)$ and $F_n(\tau)\to F_n(\widetilde\tau)$, respectively. Slightly abusing the notation, we will still refer to the restricted maps by $\alpha$ and $\beta$.

\begin{proof}[Proof of \cref{thm:westbijection}]
Let $F\in F_n(\tau)$; we claim $\alpha(\beta(F))=F$. It suffices to show that for any special vertex $v$ of $F$, if we antishuffle each special strict ancestor of $v$, in breadth-first search order, then the label of $v$ becomes the largest among its descendants. If this is true, then each shuffle in $\alpha$ exactly undoes one antishuffle in $\beta$. But since $F$ initially avoided $\tau$, this was true in $F$, and the relative order of the labels of the subtree rooted at $v$ is preserved by each shuffle, so it remains true.

Similarly, to show that $\beta(\alpha(F))=F$ for any $F\in F_n(\widetilde\tau)$, we show that for any special vertex $v$ of $F$, if we shuffle each special strict descendant of $v$, in reverse breadth-first search order, then the label of $v$ becomes the smallest label $y$ among the label set of the subtree rooted at $v$ such that $v$ would still be special if $v$ were labeled with $y$. Since $F$ initially avoided $\widetilde\tau$, this must have been true in $F$, and then since both the label of $v$ and the label set of the subtree rooted at $v$ have not changed, it must still be true. Thus, $\alpha,\beta$ are inverse maps.
\end{proof}

Note that up to complementation, this is similar to the method applied in \cite{andersarcher} for the case $\tau=123,\widetilde\tau=132$. However, using our terminology, the $\alpha$ given by Anders and Archer shuffles special vertices in the usual breadth-first search order, instead of the reversed order, as is done above.

\subsection{Restricting the bijection}\label{subsec:restrict-west-bijection}
Using the notation from \cref{subsec:west-bijection-proof}, we again consider the maps $\alpha\colon F_n(\widetilde\tau)\to F_n(\tau)$ and $\beta\colon F_n(\tau)\to F_n(\widetilde\tau)$. We first prove the following result:

\begin{proposition}\label{prop:forest-wilf-inequality}
Let $m\ge 2$, and let $\sigma\in\mathcal{S}_m$ be such that $\sigma(m)=m$. The restriction of $\alpha$ to $F_n(\sigma,\widetilde\tau)$ yields an injection $F_n(\sigma,\widetilde\tau)\to F_n(\sigma,\tau)$, so $f_n(\sigma,\widetilde\tau)\le f_n(\sigma,\tau)$ for all $n$.
\end{proposition}

\begin{proof}
Using the fact that $\alpha,\beta$ are inverse maps between $F_n(\widetilde\tau)$ and $F_n(\tau)$, which is proven in the proof of \cref{thm:westbijection}, it suffices to show that if $F\in F_n(\tau)$ contains $\sigma$, then so does $\beta(F)$. To see this, choose vertices $v_1,\dots,v_m$ in $F$ such that $v_i$ is a strict ancestor of $v_j$ if $i<j$ and the labels of the vertices $v_1,\dots,v_m$ are in the same relative order as $\sigma$. Note that $v_i$ is not special for all $1\le i\le m-1$, since each such vertex has a label smaller than that of $v_m$. Thus, each of these vertices has the same label in $\beta(F)$ as in $F$.

Now during the construction of $\beta(F)$ from $F$, we see that the label of the vertex $v_m$ does not decrease while we antishuffle each strict special ancestor of $v_m$, finally achieving some label $x$ after we antishuffle the special strict ancestor of $v_m$ that is closest to $v_m$. Then for the rest of the construction of $\beta(F)$, the label $x$ stays in the subtree rooted at $v_m$. Combining the corresponding vertex with vertices $v_1,\dots,v_{m-1}$, we find a sequence of vertices of $\beta(F)$ whose labels have relative order $\sigma$.
\end{proof}

In the case $\sigma=213$, we also find surjectivity:
\begin{theorem}\label{thm:213-equivalences}
The restriction of $\alpha$ to $F_n(213,\widetilde\tau)$ and the restriction of $\beta$ to $F_n(213,\tau)$ give inverse maps $\alpha\colon F_n(213,\widetilde\tau) \to F_n(213,\tau)$ and $\beta\colon F_n(213,\tau) \to F_n(213,\widetilde\tau)$.
\end{theorem}
\begin{remark}
Note that the only nontrivial forest-Wilf equivalence the theorem gives is in the case
\[\tau=12\cdots (k-1)k\quad\text{and}\quad\widetilde\tau=12\cdots k(k-1).\]
In the case $k=3$, we have already seen in \cref{subsubsec:213-123-recurrence} that $f_n(213,123)$ and $f_n(213,132)$ satisfy the same recurrence.
\end{remark}
\begin{proof}
After using \cref{prop:forest-wilf-inequality}, it suffices to show that if $F\in F_n(\widetilde\tau)$ contains $213$, then so does $\alpha(F)$. Define $\bar\tau$ and special vertices in the same manner as in \cref{subsec:west-bijection-proof}.

Suppose $F\in F_n(\widetilde\tau)$ contains $213$. Let $v_1,v_2,v_3$ be an instance of $213$ in $F$, so that $v_1$ is an ancestor of $v_2$, which is an ancestor of $v_3$. If $v_1$ is special, then there exists a sequence $u_1,\dots,u_{k-2},v_1$ establishing that $v_1$ is special. Let $u_i$ have the largest label among $u_1,\dots,u_{k-2}$; since $F\in F_n(\widetilde\tau)$, the label of $v_2$ must be less than the label of $u_i$. Then, we can replace $v_1$ with $u_i$ to get another instance of $213$ with $v_1$ strictly older. This process can only be repeated finitely many times, so we may assume $v_1$ is not special. After fixing this $v_1$, by possibly changing $v_2$, we may assume that all vertices on the shortest path between $v_1$ and $v_2$ have a label that is larger than the label of $v_1$. This then implies that $v_2$ is not special, as otherwise $v_1$ would also be special.

Let $y$ be the label of $v_3$ in $F$. During each step of the construction of $\alpha(F)$, the labels of $v_1$ and $v_2$ remain the same. After we shuffle each special descendant of $v_2$, the label $y$ corresponds to some vertex $v_3'$ in the subtree rooted at $v_2$. By our assumption on $v_2$, we find that as we shuffle each special vertex on the shortest path between $v_1$ and $v_2$, the label of $v_3'$ remains greater than the label of $v_1$. Then, after we shuffle each special ancestor of $v_1$, the relative order of the labels of $v_1,v_2,v_3'$ remains the same. Thus, $\alpha(F)$ contains $213$.
\end{proof}

\section{Forest-shape-Wilf equivalences}\label{sec:forest-shape-wilf-equivalences}
Here we adapt the definitions and methods used regarding shape-Wilf equivalence in \cite{bwx} to define a relation we call \textit{forest-shape-Wilf equivalence}, which is stronger than forest-structure-Wilf equivalence. We then prove \cref{thm:fswe}. For background and motivation, we first define shape-Wilf equivalence.

Let $Y$ be a Young diagram (which we draw using the \textit{English} convention). We view $Y$ as a finite set of ordered pairs $(r,c)$, where $r$ and $c$ are positive integers, such that if $(r,c)\in Y$, then $(r',c')\in Y$ for all positive integers $r',c'$ with $r'\le r$ and $c'\le c$. Visually, the index $r$ corresponds to row number, and $c$ corresponds to column number. A \textit{transversal} $T$ of $Y$ is a labeling of the members of $Y$ with $0$'s and $1$'s such that each row and column contains exactly one $1$; in other words, if $(r_0,1)\in Y$, then $(r_0,c)$ is labeled $1$ for exactly one value of $c$, and if $(1,c_0)\in Y$, then $(r,c_0)$ is labeled $1$ for exactly one value of $r$. See \cref{fig:youngdiagramtransversal} for an example of a transversal.

\begin{figure}
\ytableausetup{centertableaux,boxsize=2em}
\begin{ytableau}
 0 & 0 & 0 & 0 & 0 & 0 & 1 & 0 \\
 0 & 0 & 0 & 0 & 0 & 0 & 0 & 1\\
 0 & 0 & 0 & 0 & 1 & 0 & 0\\
 1 & 0 & 0 & 0 & 0 & 0 \\
 0 & 0 & 0 & 1 & 0 & 0 \\
 0 & 0 & 0 & 0 & 0 & 1 \\
 0 & 0 & 1 & 0 \\
 0 & 1
\end{ytableau}
\caption{A transversal of a Young diagram.}
\label{fig:youngdiagramtransversal}
\end{figure}

Let $k$ be a positive integer, and let $M$ be a permutation matrix of size $k\times k$. We adopt standard matrix conventions, so that row number increases downward and column number increases to the right. We say that a transversal $T$ of the Young diagram $Y$ \textit{contains} $M$ if there exist positive integers $r_1<\cdots<r_k$ and $c_1<\cdots<c_k$ such that for all $1\le i,j\le k$, we have $(r_i,c_j)\in Y$ and that the label of $(r_i,c_j)$ equals the entry in the $i$th row and $j$th column of $M$. Otherwise, we say $Y$ \textit{avoids} $M$. For example, the transversal shown in \cref{fig:youngdiagramtransversal} avoids the matrix
\[\begin{bmatrix}0&0&1\\0&1&0\\1&0&0\end{bmatrix}.\]
But if we add the cell $(4,7)$ to the Young diagram and label it $0$, then the transversal will contain that matrix.

Two sets $\mathcal{M}$ and $\mathcal{M}'$ of permutation matrices (not necessarily all the same size) are \textit{shape-Wilf equivalent} if for all Young diagrams $Y$, the number of transversals avoiding $\mathcal{M}$ equals the number of transversals avoiding $\mathcal{M}'$ \cite{bwx}.

Our definition of forest-shape-Wilf equivalence is a natural generalization of shape-Wilf equivalence, in much the same way that forest-Wilf equivalence is a generalization of Wilf equivalence (however, forest-shape-Wilf equivalence is not necessarily stronger than shape-Wilf equivalence; see \cref{rmk:forestshapevsshape}). But we must first define an analog of Young diagrams.

\begin{definition}
Given a rooted forest $F$ (without numerical labels), we define a \textit{forest-Young diagram} on $F$ to be a finite set $Y$ of ordered pairs each of the form $(r,v)$, where $r$ is a positive integer and $v$ is a vertex of $F$. We require that $Y$ satisfy the following three conditions:
\begin{itemize}
    \item $(1,v)\in Y$ for all vertices $v$.
    \item If $(r,v)\in Y$, then $(r',v)\in Y$ for all integers $r'$ such that $1\le r'\le r$.
    \item If $(r,v)\in Y$, then $(r,v')\in Y$ for all descendants $v'$ of $v$.
\end{itemize}
\end{definition}

\begin{remark}\label{rmk:leaf/root}
In the third condition, we made the decision to write \textit{descendants} instead of \textit{ancestors}. We could have chosen ancestors, but the definition using descendants works for our purposes in this paper. One may refer to this version as \textit{leaf-heavy forest-Young diagrams}, and the flipped version as \textit{root-heavy forest-Young diagrams}.
\end{remark}

Note that in the case that the underlying forest $F$ is just a path with $n$ vertices, forest-Young diagrams on $F$ are equivalent to Young diagrams with $n$ columns.

We can visualize $Y$ as a set of cells in three-dimensional space, where each vertex $v$ lies above its own \textit{column} of cells, with $r$ the \textit{row number} of the cell $(r,v)$. Here, for a fixed vertex $v_0$, the set of ordered pairs in $Y$ of the form $(r,v_0)$ is the column corresponding to $v_0$. Similarly, for a fixed row number $r_0$, the set of ordered pairs in $Y$ of the form $(r_0,v)$ is the $r_0$th \textit{row}. Note that the conditions dictate that the columns are top-aligned. In keeping with the visualization, we say the cell $(r,v)$ is \textit{above} the cell $(r',v')$ if $r<r'$, and \textit{below} the cell $(r',v')$ if $r>r'$. Similarly, we say $(r,v)$ is \textit{younger} than $(r',v')$ if $v$ is a strict descendant of $v'$, and $(r,v)$ is \textit{older} than $(r',v')$ if $v$ is a strict ancestor of $v'$. For an example of a forest-Young diagram, see \cref{fig:forestyoungdiagramexample}.

\begin{figure}
\includegraphics[height=150pt]{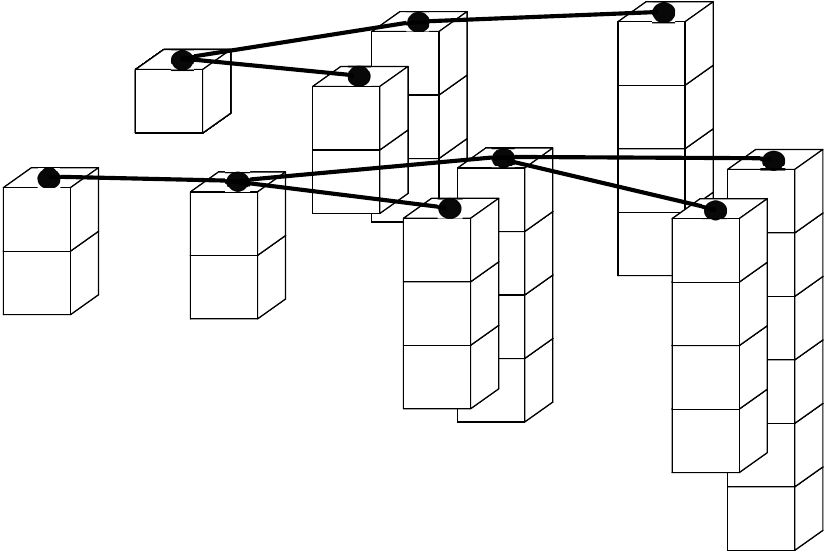}
\caption{A forest-Young diagram.}
\label{fig:forestyoungdiagramexample}
\end{figure}

\begin{definition}
A \textit{transversal} $T$ of a forest-Young diagram $Y$ is a labeling of the members of $Y$ with $0$'s and $1$'s such that each (nonempty) row and column contains exactly one $1$. Let $S_Y$ be the set of all transversals of $Y$.
\end{definition}
If the forest underlying $Y$ contains $n$ vertices, then there will be exactly $n$ cells labeled $1$ in a transversal of $Y$. So, we can think of a transversal of $Y$ as a sort of ``generalized labeling'' of the vertices of the forest $F$, where each vertex $v$ of $F$ is ``labeled'' with the unique $r$ such that $(r,v)$ is labeled $1$ in $Y$. Note that forest-Young diagrams do not necessarily always contain transversals; for instance, it is necessary that $r\le n$ for all $(r,v)\in Y$, though this is in general not sufficient.

\begin{definition}\label{def:forest-young-transversal-containment}
Let $M$ be a permutation matrix of size $k\times k$. A transversal $T$ of the forest-Young diagram $Y$ is said to \textit{contain} the matrix $M$ if there exists a sequence $v_1,\dots,v_k$ of vertices of $F$ and a sequence $r_1<\cdots<r_k$ of row indices such that the following conditions hold:
\begin{itemize}
    \item The vertex $v_i$ is a strict ancestor of $v_j$ if $i<j$.
    \item For all $1\le i,j\le k$, we have $(r_i,v_j)\in Y$.
    \item For all $1\le i,j\le k$, the label of $(r_i,v_j)$ equals the entry in the $i$th row and $j$th column of $M$.
\end{itemize}
Such a collection of cells $\{(r_i,v_j)\}$ is an \textit{instance} of $M$ in $T$. If $T$ does not contain $M$, we say $T$ \textit{avoids} $M$.
\end{definition}
Note the significance of the second condition: each $1$ \textit{and} each $0$ of $M$ must correspond to some cell in $Y$.

Given a forest-Young diagram $Y$ and a set $\mathcal{M}$ of permutation matrices $M$, let $S_Y(\mathcal{M})$ be the set of all transversals of $Y$ that avoid each $M\in\mathcal{M}$. For convenience, if $M$ is a single permutation matrix, we often write $S_Y(M)$ in place of $S_Y(\{M\})$.

\begin{definition}
We say two sets $\mathcal{M}$ and $\mathcal{M}'$ of permutation matrices are \textit{forest-shape-Wilf equivalent} if $|S_Y(\mathcal{M})|=|S_Y(\mathcal{M}')|$ for all forest-Young diagrams $Y$.
\end{definition}

\begin{remark}
In accordance with \cref{rmk:leaf/root}, our definition of forest-shape-Wilf equivalence may perhaps more properly be referred to as \textit{leaf-heavy-forest-shape-Wilf equivalence}, as we are using leaf-heavy forest-Young diagrams. It is possible that some sets of matrices are instead \textit{root-heavy-forest-shape-Wilf equivalent}, but we do not currently know any examples of such equivalences.
\end{remark}

We identify a permutation $\sigma=\sigma(1)\cdots\sigma(k)\in\mathcal{S}_k$ with the $k\times k$ permutation matrix $M$, where the entry in the $i$th row and $j$th column of $M$ equals $\delta_{i,n+1-\sigma(j)}$. This allows us to keep the resemblance with the ``shape'' of $\sigma$, which is the convention adopted in \cite{sw}. Note this differs from the convention in \cite{bwx}. In view of this correspondence between permutations and permutation matrices, we note the following:

\begin{lemma}
Forest-shape-Wilf equivalence implies forest-structure-Wilf equivalence.
\end{lemma}

\begin{proof}
Fix a forest $F$ on $[n]$ and apply the definition of forest-shape-Wilf equivalence to the forest-Young diagram $Y$ that consists of all $n^2$ ordered pairs of the form $(r,v)$, where $1\le r\le n$ and $v$ is a vertex of $F$. We let the row number $r$ of a member $(r,v)$ of a transversal correspond to the label of $v$ in the forest.
\end{proof}

\begin{remark}\label{rmk:forestshapevsshape}
One cannot find a similarly easy proof that forest-shape-Wilf equivalence implies shape-Wilf equivalence. However, one can show that if $\mathcal{M},\mathcal{M}'$ are forest-shape-Wilf equivalent, then their \textit{reverses}, obtained by ``reversing'' each matrix, are shape-Wilf equivalent.
\end{remark}

We are now in a position to prove \cref{thm:fswe}. The result follows easily from \cref{prop:i2j2,prop:blocks} below. We prove these using the methods established in \cite{bwx}, specializing to the $t=2$ case. Both proofs closely follow the proofs of analogous results in \cite{bwx}. The difference here is that we work with forests and forest-Young diagrams instead of Young diagrams.

First, let $I_2$ and $J_2$ be the following matrices:
\[I_2=\begin{bmatrix}1&0\\0&1\end{bmatrix},\qquad J_2=\begin{bmatrix}0&1\\1&0\end{bmatrix}.\]

\begin{proposition}\label{prop:i2j2}
The permutation matrices $I_2$ and $J_2$ are forest-shape-Wilf equivalent.
\end{proposition}

Before giving the proof, we first define two operations on transversals of a forest-Young diagram $Y$ and discuss their key properties.
\begin{definition}
Suppose $L$ is a transversal of $Y$ that contains $I_2$. Let $a_2$ be the highest cell labeled with $1$ such that $L$ contains an instance of $I_2$ in which $a_2$ is the lower $1$. Let $a_1$ be the youngest cell labeled with $1$ such that $L$ contains an instance of $I_2$ with $1$'s at $a_1$ and $a_2$, with $a_2$ being the lower $1$. Let $b_1,b_2$ be the two cells labeled $0$ in the instance of $I_2$ containing $a_1,a_2$, such that $b_1$ is lower and older than $b_2$. Relabel the cells $a_1,a_2,b_1,b_2$ with $0$'s and $1$'s so that they now form an instance of $J_2$. Let $\phi(L)$ be the resulting transversal.
\end{definition}

The following lemma gives some important properties of $\phi$.

\begin{lemma}\label{lem:phi}
Suppose the transversal $L$ contains an instance of $I_2$. Then, $\phi(L)$ does not contain an instance of $I_2$ with each cell lying above $a_2$. If in addition $L$ contains no $J_2$ with its lower cells below $a_2$, then neither does $\phi(L)$.
\end{lemma}

\begin{proof}
For the first statement, suppose otherwise. Let the $1$'s in the potential instance of $I_2$ correspond to cells $c_1$ and $c_2$, where $c_1$ is higher than $c_2$. Note that either $c_1$ or $c_2$ must equal $b_2$, as otherwise we contradict our choice of $a_2$. But $c_1\neq b_2$, as otherwise we again contradict our choice of $a_2$. Thus, $c_2=b_2$, and $c_1$ is older than $a_2$. But $c_1$ being younger than $a_1$ contradicts the choice of $a_1$, and $c_1$ being older than $a_1$ contradicts the choice of $a_2$ (where we must be careful to check each of the three conditions for our definition of containment in all cases). This proves the first statement.

We also prove the second statement via contradiction. Suppose now that $L$ contains no instance of $J_2$ with its lower cells below $a_2$, but $\phi(L)$ does. Let the $1$'s in this instance of $J_2$ in $\phi(L)$ correspond to cells $d_1$ and $d_2$, where $d_1$ is lower than $d_2$ and $a_2$. We must have $d_2=b_1$ or $d_2=b_2$, or we contradict our assumption. But in either case, the labels of $1$ on cells $d_1$ and $a_2$ in $L$ yield an instance of a valid $J_2$, which is again a contradiction. Here and later, by saying \textit{valid} we emphasize that each entry of the matrix $J_2$ (including the zero entries) corresponds to the label of some cell of $Y$, in the sense of the second condition in \cref{def:forest-young-transversal-containment}; this becomes more significant for containing instances of $I_2$, where the lower-left entry of the matrix is zero.
\end{proof}

We now define our second operation on transversals.
\begin{definition}
Let $T$ be a transversal of $Y$ that contains an instance of $J_2$. Let $b_1$ be the lowest cell labeled $1$ such that $T$ contains an instance of $J_2$ in which $b_1$ is the lower $1$. Then, let $b_2$ be the lowest cell labeled $1$ such that $T$ contains an instance of $J_2$ with $1$'s at $b_1$ and $b_2$, with $b_1$ being the lower $1$. Let $a_1,a_2$ be the two cells labeled $0$ in the instance of $J_2$ containing $b_1,b_2$, such that $a_1$ is higher and older than $a_2$. Relabel the cells $a_1,a_2,b_1,b_2$ with $0$'s and $1$'s so that they now form an instance of $I_2$. Let $\psi(T)$ be the resulting transversal.
\end{definition}

Analogously, we have the following lemma addressing $\psi$.

\begin{lemma}\label{lem:psi}
Suppose the transversal $T$ contains an instance of $J_2$. Then, $\psi(T)$ does not contain an instance of $J_2$ with its lower cells below $a_2$. If in addition $T$ contains no $I_2$ with each cell lying above $a_2$, then neither does $\psi(T)$.
\end{lemma}

\begin{proof}
Again, we prove each statement by contradiction. For the first statement, suppose $\psi(T)$ does contain such an instance of $J_2$, with the $1$'s corresponding to cells $e_1$ and $e_2$, where $e_1$ is lower than $e_2$ and $a_2$. We must have $e_2=a_1$ or $e_2=a_2$. In either case, the labels of $1$ on cells $e_1$ and $b_2$ in $T$ give us a valid instance of $J_2$, which contradicts our choice of $b_1$.

For the second statement, suppose $\psi(T)$ contains an instance of $I_2$ with each cell lying above $a_2$. Let the $1$'s of this $I_2$ correspond to cells $f_1$ and $f_2$, where $f_1$ is higher than $f_2$. Note we must have $f_1=a_1$ or $f_2=a_1$. In the first case, the $1$'s at $b_1$ and $f_2$ in $T$ yield a valid instance of $J_2$, contradicting our choice of $b_2$. In the second case, the $1$'s at $f_1$ and $b_2$ yield a valid instance of $I_2$ in $T$ with each cell lying above $a_2$, another contradiction.
\end{proof}

Now we can define the maps $\theta,\eta$ that will become our inverse maps for proving \cref{prop:i2j2}.

\begin{definition}\label{def:i2j2-theta-eta}
Define the map $\theta\colon S_Y(J_2)\to S_Y(I_2)$ as follows: given a transversal $L\in S_Y(J_2)$, let $\theta(L)$ be the result of iteratively applying $\phi$ to $L$ until the resulting transversal avoids $I_2$. Similarly, define the map $\eta\colon S_Y(I_2)\to S_Y(J_2)$ by iteratively applying $\psi$ to $T\in S_Y(I_2)$ until the resulting transversal avoids $J_2$.
\end{definition}
Note that the process defining $\theta$ must terminate, since by the first statement of \cref{lem:phi}, at each iteration of $\phi$, the cell $a_2$ strictly increases its row number (gets lower) or keeps the same row number but gets older. Similarly, the process defining $\eta$ also terminates, since using \cref{lem:psi} we see that at each iteration of $\psi$, the cell $b_1$ strictly decreases its row number (gets higher) or keeps the same row number but gets younger.

We can now prove \cref{prop:i2j2}.
\begin{proof}[Proof of \cref{prop:i2j2}]
Define $\theta,\eta$ as in \cref{def:i2j2-theta-eta}. It suffices to show that $\eta(\theta(L))=L$ for all $L\in S_Y(J_2)$ and $\theta(\eta(T))=T$ for all $T\in S_Y(I_2)$.

Let $L\in S_Y(J_2)$; we first show that $\eta(\theta(L))=L$. Suppose it takes $N$ applications of $\phi$ to $L$ to reach $\theta(L)$; that is, $\phi^N(L)=\theta(L)$. It suffices to show that for all integers $n$ with $1\le n\le N$, we have $\psi(\phi(\phi^{n-1}(L)))=\phi^{n-1}(L)$. Note that since $L$ avoids $J_2$, we may induct using \cref{lem:phi} to see that for each $1\le n\le N$, before and after applying $\phi$ to $\phi^{n-1}(L)$, there is no instance of $J_2$ with its lower cells below $a_2$. We will show that the instance of $J_2$ created by applying $\phi$ to $\phi^{n-1}(L)$ is the one identified when applying $\psi$ to $\phi^n(L)$. Fix an $n$ with $1\le n\le N$. Choose cells $a_1,a_2,b_1,b_2$ according to when applying $\phi$ to $\phi^{n-1}(L)$. Using the inductive result, we see that $\psi$ chooses the cell $b_1$ correctly (meaning that the choice of $b_1$ corresponds to the instance of $J_2$ identified earlier), as this $b_1$ is a candidate and there are no valid candidates below it. Now suppose $\psi$ chooses its $b_2$ incorrectly, choosing instead some cell $b_2'\neq b_2$ which must be lower than $b_2$ and higher and younger than $b_1$. But note that then the $1$'s at cells $a_1$ and $b_2'$ in $\phi^{n-1}(L)$ yield a valid instance of $I_2$, contradicting the choice of $a_2$. Thus, $\psi$ also chooses $b_2$ correctly.

Let $T\in S_Y(I_2)$. The proof that $\theta(\eta(T))=T$ is similar. Suppose it takes $N'$ applications of $\psi$ to $T$ to reach $\eta(T)$. Again, we show that for $n'$ such that $1\le n'\le N'$, we have that $\phi(\psi(\psi^{n'-1}(T)))=\psi^{n'-1}(T)$. Fix such an $n'$; we show that the instance of $I_2$ created by applying $\psi$ to $\psi^{n'-1}(T)$ is the instance of $I_2$ identified by $\phi$ when applied to $\psi^{n'}(T)$. Choose cells $b_1,b_2,a_1,a_2$ according to when applying $\psi$ to $\psi^{n'-1}(T)$. By a similar inductive argument as above but using \cref{lem:psi}, we see that neither $\psi^{n'-1}(T)$ nor $\psi(\psi^{n'-1}(T))=\psi^{n'}(T)$ contains an instance of $I_2$ with each cell lying above $a_2$; thus the application of $\phi$ chooses $a_2$ correctly. Suppose $\phi$ chooses the cell $a_1$ incorrectly, instead choosing $a_1'$, which must be older and higher than $a_2$ and younger than $a_1$. If $a_1'$ is below $a_1$, then the $1$'s at cells $b_1$ and $a_1'$ form a valid instance of $J_2$ in $\psi^{n'-1}(T)$, contradicting the choice of $b_2$. But if $a_1'$ is above $a_1$, then the $1$'s at $a_1'$ and $b_2$ in $\psi^{n'-1}(T)$ yield a valid instance of $I_2$ lying completely above $a_2$, a contradiction. Thus, the application of $\phi$ chooses $a_1$ correctly as well.
\end{proof}

We now move on to the second proposition required to prove \cref{thm:fswe}.

\begin{proposition}\label{prop:blocks}
Suppose the permutation matrices $C$ and $D$ are forest-shape-Wilf equivalent. Let $\mathcal{A}=\{A_1,\dots,A_m\}$ be a set of permutation matrices. For each $i$ with $1\le i\le m$, define the matrices
\[M_i=\begin{bmatrix}0&C\\A_i&0\end{bmatrix},\qquad M_i'=\begin{bmatrix}0&D\\A_i&0\end{bmatrix}.\]
Then the sets $\mathcal{M}=\{M_1,\dots,M_m\}$ and $\mathcal{M}'=\{M_1',\dots,M_m'\}$ are forest-shape-Wilf equivalent.
\end{proposition}

The proof comes in a few steps. We first make a preliminary definition:
\begin{definition}\label{def:(A,L)-coloring}
Let $L$ be a transversal of a forest-Young diagram $Y$, and let $\mathcal{A}=\{A_1,\dots,A_m\}$ be a set of permutation matrices. The \textit{$(\mathcal{A},L)$-coloring of $Y$} is a coloring of the cells of $Y$ that is constructed as follows.
\begin{enumerate}
    \item For each cell $(r,v)$ of $Y$, color $(r,v)$ white if the transversal $L$ contains an instance of some $A_i\in\mathcal{A}$ in which each cell is older and lower than $(r,v)$. Otherwise, color $(r,v)$ blue.
    \item For each cell labeled $1$ that is colored blue, color the entire column and row of the cell blue as well.
\end{enumerate}
The white cells in the $(\mathcal{A},L)$-coloring of $Y$ can be naturally reassembled into a forest-Young diagram $Y_{\mathcal{A},L}$, with the remaining $1$'s forming a transversal $L_\mathcal{A}$. Then, $Y_{\mathcal{A},L}$ and $L_\mathcal{A}$ can be naturally viewed as subsets of $Y$ and $L$, respectively.
\end{definition}
We now make precise the construction of $Y_{\mathcal{A},L}$ and $L_\mathcal{A}$. First we note that if the cell $c$ is colored white in Step 1, then so is each cell that is above and younger than $c$. Thus, the induced subgraph of $F$ formed from the set of vertices $v_0$ such that there is some $(r,v_0)\in Y$ colored white in Step 1 forms a rooted forest $F_{1,(\mathcal{A},L)}$. The set of cells colored white in Step 1 then forms a forest-Young diagram $Y_{1,(\mathcal{A},L)}$ above $F_{1,(\mathcal{A},L)}$. So, the set of cells of $Y_{1,(\mathcal{A},L)}$ that are labeled $1$ form a ``partial transversal'' of $Y_{1,(\mathcal{A},L)}$: each row and column contains at most one cell labeled $1$.

Now in Step 2, we color some of the white cells blue: we can first delete those vertices $v_1$ of $F_{1,(\mathcal{A},L)}$ for which there is a blue cell $(r,v_1)\in Y$ labeled $1$, and color all their corresponding cells blue. Then, the remaining vertices can be made to form a rooted forest $F_{2,(\mathcal{A},L)}$, in which for each remaining vertex, it becomes a root if it has no remaining ancestor, and otherwise its closest remaining ancestor from $F_{1,(\mathcal{A},L)}$ becomes its parent. The cells which remain white still form a forest-Young diagram $Y_{2,(\mathcal{A},L)}$ with respect to $F_{2,(\mathcal{A},L)}$.

Finally, for the second part of Step 2, if for row $r_0$ there exists a cell $(r_0,v)\in Y$ labeled $1$ that was colored blue in Step 1, we color each cell in that row blue. This may delete some vertices of $F_{2,(\mathcal{A},L)}$, but since all ancestors of a deleted vertex are also deleted, the remaining vertices form a rooted forest, denoted $F_{\mathcal{A},L}$. Along with the column deletion resulting from deleted vertices of $F_{2,(\mathcal{A},L)}$, this may also delete some rows of $Y_{2,(\mathcal{A},L)}$. So, we shift all remaining cells of $Y_{2,(\mathcal{A},L)}$ upward, re-indexing accordingly. Thus, we obtain a forest-Young diagram $Y_{\mathcal{A},L}$ based on the rooted forest $F_{\mathcal{A},L}$. The cells in $Y_{\mathcal{A},L}$ labeled $1$ necessarily form a transversal $L_\mathcal{A}$ of $Y_{\mathcal{A},L}$. Note we can view $Y_{\mathcal{A},L}$ as a subset of $Y$ and $L_\mathcal{A}$ as a subset of $L$, as claimed.

For the remainder of this section, let $C,D,\mathcal{A}=\{A_1,\dots,A_m\},\mathcal{M}=\{M_1,\dots,M_m\},\mathcal{M}'=\{M_1',\dots,M_m'\}$ be as in the statement of \cref{prop:blocks}, and fix a forest-Young diagram $Y$.

\begin{lemma}\label{lem:subtransversal-avoidance}
Let $L\in S_Y(\mathcal{M})$, and let $Y_{\mathcal{A},L}$ and $L_\mathcal{A}$ be the forest-Young diagram and transversal, respectively, resulting from the $(\mathcal{A},L)$-coloring of $Y$. Then, $L_\mathcal{A}\in S_{Y_{\mathcal{A},L}}(C)$. Similarly, given $T\in S_Y(\mathcal{M}')$, we have $T_\mathcal{A}\in S_{Y_{\mathcal{A},T}}(D)$, where $Y_{\mathcal{A},T}$ and $T_{\mathcal{A}}$ are obtained from the $(\mathcal{A},T)$-coloring of $Y$.
\end{lemma}

\begin{proof}
Suppose $L_\mathcal{A}$ contains $C$. Then, viewing $Y_{\mathcal{A},L}$ as a subset of $Y$, when constructing the $(\mathcal{A},L)$-coloring of $L$, each cell in an instance of $C$ was colored white in Step 1, so we can combine the $C$ with some $A_i\in\mathcal{A}$ lower and older than it to obtain an instance of $M_i\in\mathcal{M}$ in $L$, a contradiction. The proof that $T_\mathcal{A}\in S_{Y_{\mathcal{A},T}}(D)$ is the same.
\end{proof}

For the rest of this section, for each forest-Young diagram $Z$, we fix a bijection $\Pi_Z\colon S_Z(C)\to S_Z(D)$, with inverse map $\Pi_Z^{-1}\colon S_Z(D)\to S_Z(C)$. The existence of $\Pi_Z$ follows from the assumed forest-shape-Wilf equivalence of $C$ and $D$. We use these maps to define two functions $\alpha$ and $\beta$, which will become the inverse maps we use to prove \cref{prop:blocks}.

\begin{definition}\label{def:blocks-alpha-beta}
We define a map $\alpha\colon S_Y(\mathcal{M})\to S_Y$ as follows. Given $L\in S_Y(\mathcal{M})$, we construct the $(\mathcal{A},L)$-coloring of $Y$ and obtain a forest-Young diagram $Y_{\mathcal{A},L}$ with a transversal $L_\mathcal{A}$, which we view as subsets of $Y$ and $L$, respectively. Then, we modify $L$ by replacing the transversal $L_\mathcal{A}$ of $Y_{\mathcal{A},L}$ with the transversal $\Pi_{Y_{\mathcal{A},L}}(L_\mathcal{A})$. Let $\alpha(L)$ be the resulting transversal.

Define a map $\beta\colon S_Y(\mathcal{M}')\to S_Y$ as follows. Given $T\in S_Y(\mathcal{M})$, we construct the $(\mathcal{A},T)$-coloring of $Y$ and obtain a forest-Young diagram $Y_{\mathcal{A},T}$ with a transversal $T_\mathcal{A}$, which we view as subsets of $Y$ and $T$, respectively. Then, we modify $T$ by replacing the transversal $T_\mathcal{A}$ of $Y_{\mathcal{A},T}$ with the transversal $\Pi_{Y_{\mathcal{A},T}}^{-1}(T_\mathcal{A})$. The resulting transversal of $Y$ is $\beta(T)$.
\end{definition}

Note that the expressions $\Pi_{Y_{\mathcal{A},L}}(L_\mathcal{A})$ and $\Pi_{Y_{\mathcal{A},T}}^{-1}(T_\mathcal{A})$ are well-defined by \cref{lem:subtransversal-avoidance}. One may view $\alpha(L)$ as essentially $(L\setminus L_\mathcal{A})\cup \Pi_{Y_{\mathcal{A},L}}(L_\mathcal{A})$, and similarly $\beta(T)$ as essentially $(T\setminus T_\mathcal{A})\cup \Pi_{Y_{\mathcal{A},T}}^{-1}(T_\mathcal{A})$.

\begin{lemma}\label{lem:L-alpha(L)-same-coloring}
Let $L\in S_Y(\mathcal{M})$. Then, the $(\mathcal{A},L)$-coloring of $Y$ is the same as the $(\mathcal{A},\alpha(L))$-coloring of $Y$. Similarly, given $T\in S_Y(\mathcal{M}')$, the $(\mathcal{A},T)$-coloring of $Y$ is the same as the $(\mathcal{A},\beta(T))$-coloring of $Y$. 
\end{lemma}

\begin{proof}
Let $L\in S_Y(\mathcal{M})$. We will show that at each step in the constructions of the $(\mathcal{A},L)$-coloring and $(\mathcal{A},\alpha(L))$-coloring of $Y$, the cells of $Y$ are colored exactly the same way. Suppose cell $c\in Y$ is colored blue in Step 1 of the $(\mathcal{A},L)$-coloring. Then, the cells lower and older than $c$ are colored blue as well, so their labels are the same for $L$ and for $\alpha(L)$. Thus, $\alpha(L)$ will also not contain any instance of $A_i$ completely lower and older than $c$, so $c$ will be colored blue in Step 1 of the $(\mathcal{A},\alpha(L))$-coloring.

Now suppose $c=(r,v)\in Y$ is colored white in Step 1 of the $(\mathcal{A},L)$-coloring, so for some $A_i\in\mathcal{A}$ there is an instance of $A_i$ in $L$ with all its cells older and lower than $c$. Let $v_p$ be the parent of $v$, if it exists. Without loss of generality, we may assume that each of the ordered pairs $(r+1,v)$ and $(r,v_p)$ either does not exist (here we say that $(r,v_p)$ does not exist if $v_p$ is not defined), is not a member of $Y$, or is colored blue in Step 1 of the $(\mathcal{A},L)$-coloring. Then, we can repeat the same argument; the cells lower and older than $c$ are colored blue, so their labels are unaffected when we apply $\alpha$; hence, $\alpha(L)$ will still an instance of $A_i$ lower and older than $c$, so $c$ is colored white in Step 1 of the $(\mathcal{A},\alpha(L))$-coloring. This shows that the cells will be colored the same way in Step 1 of the $(\mathcal{A},\alpha(L))$-coloring as they are in Step 1 of the $(\mathcal{A},L)$-coloring. Then, since the set of blue cells labeled $1$ is the same after Step 1 of both colorings, in Step 2 we also color the same rows and columns blue.

Similarly, we can show that the $(\mathcal{A},T)$-coloring is the same as the $(\mathcal{A},\beta(T))$-coloring.
\end{proof}

\begin{lemma}\label{lem:blocks-alpha-beta-avoidance}
If $L\in S_Y(\mathcal{M})$, then $\alpha(L)\in S_Y(\mathcal{M}')$. Similarly, if $T\in S_Y(\mathcal{M}')$, then $\beta(T)\in S_Y(\mathcal{M})$.
\end{lemma}

\begin{proof}
Let $L\in S_Y(\mathcal{M})$, and suppose that the transversal $\alpha(L)$ contains $M_i'$ for some $1\le i\le m$. Each cell corresponding to an entry of the submatrix $D$ of $M_i'$ in an instance of $M_i'$ must then be colored white after Steps 1 and 2 of the $(\mathcal{A},\alpha(L))$-coloring. But by \cref{lem:L-alpha(L)-same-coloring}, this implies that $\Pi_{Y_{\mathcal{A},L}}(L_\mathcal{A})$ contains $D$, a contradiction. Thus, $\alpha(L)\in S_Y(\mathcal{M}')$, as claimed. Similarly $\beta(T)\in S_Y(\mathcal{M})$ if $T\in S_Y(\mathcal{M}')$.
\end{proof}

We are now in a position to prove \cref{prop:blocks}.

\begin{proof}[Proof of \cref{prop:blocks}]
Fix a forest-Young diagram $Y$, and define $\alpha,\beta$ as in \cref{def:blocks-alpha-beta}. By \cref{lem:blocks-alpha-beta-avoidance}, we may abuse notation and write $\alpha$ and $\beta$ as maps $\alpha\colon S_Y(\mathcal{M})\to S_Y(\mathcal{M}')$ and $\beta\colon S_Y(\mathcal{M}')\to S_Y(\mathcal{M})$. We claim that $\alpha$ and $\beta$ are inverse maps, which finishes the proof.

The fact that $\beta(\alpha(L))=L$ for all $L\in S_Y(\mathcal{M})$ follows easily from \cref{lem:L-alpha(L)-same-coloring}, since the obtained forest-Young diagram $Y_{\mathcal{A},L}$ in the $(\mathcal{A},L)$-coloring is the same subset of $Y$ as the forest-Young diagram $Y_{\mathcal{A},\alpha(L)}$ obtained in the $(\mathcal{A},\alpha(L))$-coloring. Similarly, we find $\alpha(\beta(T))=T$ for all transversals $T\in S_Y(\mathcal{M}')$, so we are done.
\end{proof}

We can now prove \cref{thm:fswe}.
\begin{proof}[Proof of \cref{thm:fswe}]
Fix a positive integer $m$ and patterns $\tau_i,\widetilde\tau_i\in\mathcal{S}_{k_i}$ as described in the hypotheses of the theorem statement. Utilizing the correspondence between permutations and permutation matrices, each permutation pattern $\tau_i\in\mathcal{S}_{k_i}$ corresponds to a permutation matrix $M_i$ of the form $\left[\begin{smallmatrix}0&J_2\\A_i&0\end{smallmatrix}\right]$. The corresponding pattern $\widetilde\tau_i\in\mathcal{S}_{k_i}$ then corresponds to the matrix $M_i'=\left[\begin{smallmatrix}0&I_2\\A_i&0\end{smallmatrix}\right]$. By definition, the question of the forest-shape-Wilf equivalence of $\{\tau_1,\dots,\tau_m\}$ and $\{\widetilde\tau_1,\dots,\widetilde\tau_m\}$ is the same as the question of the forest-shape-Wilf equivalence of the sets $\{M_1,\dots,M_m\}$ and $\{M_1',\dots,M_m'\}$, which holds by combining \cref{prop:i2j2} with \cref{prop:blocks}.
\end{proof}

\section{Consecutive pattern avoidance in forests}\label{sec:consecutive}
A \textit{consecutive instance} of a pattern $\sigma=\sigma(1)\cdots\sigma(k)\in\mathcal{S}_k$ in a permutation $\pi=\pi(1)\cdots\pi(n)\in\mathcal{S}_n$ is a consecutive subsequence $\pi(i)\pi(i+1)\cdots\pi(i+k-1)$ of length $k$ of $\pi$ that is in the same relative order as $\sigma=\sigma(1)\cdots\sigma(k)$. We can generalize this notion to arbitrary sequences of distinct positive integers: a \textit{consecutive instance} of the pattern $\sigma\in\mathcal{S}_k$ in a sequence is a consecutive subsequence of length $k$ that is in the same relative order as $\sigma$. Similarly, a \textit{consecutive instance} of a pattern $\sigma\in\mathcal{S}_k$ in an unordered rooted labeled forest $F$ is a sequence $v_1,\dots,v_k$ of vertices of $F$ such that $v_i$ is the parent of $v_{i+1}$ for $1\le i\le k-1$ and the labels of $v_1,\dots,v_k$ are in the same relative order as $\sigma(1)\cdots\sigma(k)$. We say that $v_1$ is the \textit{starting point} of this consecutive instance, and that $v_k$ is its \textit{endpoint}.

If a permutation, sequence, or forest has at least one consecutive instance of a pattern $\sigma$, it is said to \textit{contain} $\sigma$ (as a consecutive pattern). Otherwise, it \textit{avoids} $\sigma$ (as a consecutive pattern). A permutation, sequence, or forest is said to \textit{avoid} a set of patterns $S$ if it avoids each pattern in $S$.

The study of general consecutive pattern avoidance for permutations was begun by Elizalde and Noy in 2003 in \cite{elizaldenoy1}. Two sets of patterns $S$ and $S'$ are \textit{c-Wilf equivalent} if for all $n$, the number of length-$n$ permutations that avoid $S$ equals the number of length-$n$ permutations that avoid $S'$. However, in this paper we will focus on single-pattern sets. Two patterns $\sigma,\tau$ are \textit{strong-c-Wilf equivalent} if for all $n$ and $m$, the number of length-$n$ permutations containing exactly $m$ consecutive instances of $\sigma$ equals the number containing exactly $m$ consecutive instances of $\tau$. Clearly, strong-c-Wilf equivalence implies c-Wilf equivalence.

For rooted forests, following the definition made in \cite{andersarcher}, two sets of patterns $S$ and $S'$ are \textit{c-forest-Wilf equivalent} if for all nonnegative integers $n$, the number of forests in $F_n$ avoiding $S$ is the same as the number avoiding $S'$. For instance, clearly the sets $S$ and $S^c$ are c-forest-Wilf equivalent, where as usual $S^c$ is the set consisting of the complements of the patterns in $S$.

We say two patterns $\sigma,\tau$ are \textit{strong-c-forest-Wilf equivalent} if for all $n$ and $m$, the number of forests in $F_n$ containing exactly $m$ consecutive instances of $\sigma$ equals the number of forests in $F_n$ containing exactly $m$ consecutive instances of $\tau$. Note that strong-c-forest-Wilf equivalence implies c-forest-Wilf equivalence.

To prove our results, we will generalize the cluster method of Goulden and Jackson introduced in \cite{gouldenjackson}, used by Elizalde and Noy to analyze consecutive pattern avoidance in permutations in \cite{elizaldenoy2}. In the context of consecutive pattern avoidance in permutations, given a pattern $\sigma$, a cluster with respect to $\sigma$ is a linear overlapping set of consecutive instances of $\sigma$. The cluster numbers of $\sigma$ are the counts of clusters of a given size with a given number of consecutive instances of $\sigma$. Two patterns are strong-c-Wilf equivalent if and only if their cluster numbers are equal; this follows naturally from the Principle of Inclusion-Exclusion with the same idea as used in equations \cref{eq:F(n;1;k)_equation} and \cref{eq:f_(n;k)_equation} below. In \cref{subsec:forest-cluster-method}, we generalize this method to forests. Forests are not linear, meaning that consecutive instances of a pattern $\sigma$ can overlap in a multitude of ways, making the new clusters, which we call forest clusters, significantly more complicated. Though these difficulties arise, we prove \cref{thm:same-cluster-numbers}, which states that two patterns are strong-c-forest-Wilf equivalent if and only if their forest cluster numbers are equal; we define forest cluster numbers analogously to cluster numbers. This result also follows naturally from the Principle of Inclusion-Exclusion.

In \cref{subsec:forest-cluster-method}, after developing the forest cluster method and proving \cref{thm:same-cluster-numbers}, we prove \cref{thm:same-first-number}, which states that two patterns that are strong-c-forest-Wilf equivalent necessarily start with the same number, up to complementation. Then, in \cref{subsec:1423-1324}, we prove \cref{thm:1324-1423}, which states that the patterns $1423$ and $1324$ are strong-c-forest-Wilf equivalent. Surprisingly, $1324$ and $1423$ are not c-Wilf equivalent.

\subsection{The forest cluster method}\label{subsec:forest-cluster-method}
We now develop a generalization of the cluster method used in consecutive pattern avoidance in permutations to the setting of rooted forests. All the necessary definitions are given below.

\begin{definition}
Let $m,n\ge 1$ be integers. An \textit{$m$-forest cluster of size $n$} with respect to a pattern $\sigma$ is a labeled rooted tree with $n$ vertices and distinct positive integer labels, along with exactly $m$ distinct highlighted consecutive instances of the pattern $\sigma$ such that the following two conditions are satisfied:
\begin{itemize}
    \item Every vertex is part of some highlighted consecutive instance.
    \item It is not possible to partition the $n$ vertices into two nonempty sets such that each of the $m$ consecutive instances of $\sigma$ is completely in one set or the other.
\end{itemize}
\end{definition}
Given the first condition, the second condition is equivalent to the $m$-vertex graph $G$ being connected, where $G$ is defined as follows: the vertices of $G$ are the $m$ highlighted consecutive instances of the forest cluster, and two vertices are connected if the corresponding consecutive instances overlap (that is, they share at least one vertex).

Note that not all the possible consecutive instances of $\sigma$ need to be highlighted in a forest cluster, as demonstrated in \cref{fig:forestcluster1,fig:forestcluster2}. But when the $m$ highlighted instances are clear from context, for simplicity we often identify the forest cluster with its underlying tree. Also, the label set of the forest cluster is not required to equal $[n]=\{1,\dots,n\}$, since we only require the labels of the vertices to be distinct. This is done for convenience later.

\begin{figure}
\centering
\begin{subfigure}[t]{.5\textwidth-6.5pt}
\centering
\includegraphics{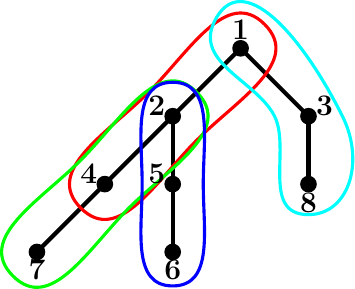}
\caption{A $4$-forest cluster of size $8$ with respect to the pattern $123$.}
\label{fig:forestcluster1}
\end{subfigure}
~\hskip 10pt
\begin{subfigure}[t]{.5\textwidth-6.5pt}
\centering
\includegraphics{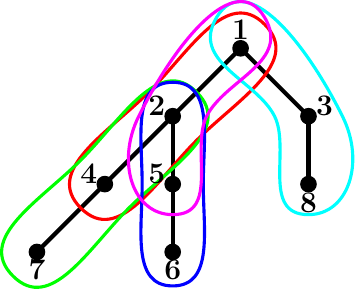}
\caption{A $5$-forest cluster of size $8$ with respect to the pattern $123$.}
\label{fig:forestcluster2}
\end{subfigure}
\caption{Two forest clusters on the same underlying tree.}
\label{fig:forestclusters}
\end{figure}

Given a pattern $\sigma$, let $r_{n, m}$ be the number of $m$-forest clusters with respect to $\sigma$ on $[n]$. We refer to the values $r_{n,m}$, which are indexed by integers $n,m\ge 1$, as the \textit{forest cluster numbers} of $\sigma$.

We first prove the following:
\begin{theorem}\label{thm:same-cluster-numbers}
Two patterns are strong-c-forest-Wilf equivalent if and only if their forest cluster numbers are equal.
\end{theorem}

\begin{proof}
We first set some notation. For integers $n\ge 1$ and $m\ge 0$, let $T(n, m)$ be the number of rooted trees on $[n]$ with $m$ highlighted consecutive instances of a pattern $\sigma$. For integers $n,i,m\ge 0$, let $F(n, i, m)$ be the number of rooted forests on $[n]$ with $m$ highlighted consecutive instances of a pattern $\sigma$, with $i$ (distinguishable) pots to put the trees in. As a reminder from \cref{sec:treeforest}, having $i$ pots means assigning each constituent tree of the forest to one of $i$ distinguishable pots, some of which may be empty. For integers $n,m\ge 0$, let $f_{n, m}$ be the number of forests on $[n]$ with exactly $m$ total consecutive instances of the pattern $\sigma$. Immediately, we have
\begin{equation}\label{eq:F(n;1;k)_equation}
F(n,1,m)=\sum_{i\ge 0}\binom{m+i}{m}f_{n,m+i}.
\end{equation}
Note this sum only has finitely many nonzero terms. On the other hand, by the Principle of Inclusion-Exclusion, we find
\begin{equation}\label{eq:f_(n;k)_equation}
f_{n,m}=\sum_{i\ge 0}(-1)^i\binom{m+i}{m}F(n,1,m+i).
\end{equation}
So to prove the statement, after applying \cref{eq:F(n;1;k)_equation} and \cref{eq:f_(n;k)_equation}, it suffices to show that the forest cluster numbers determine the values $F(n,1,m)$ in a way that is independent of $\sigma$, and vice versa. First we derive some recurrences, using methods similar to those used to count forests and trees avoiding classical patterns in \cref{sec:recurrences}.

Suppose $n\ge 1$; we will find an expression for $F(n,i,m)$. To construct a forest counted by $F(n,i,m)$, we first consider the tree containing the vertex labeled $1$, and then separate that tree from the rest of the vertices. Explicitly, suppose that the vertex labeled $1$ is in a tree with $\ell$ total vertices. This tree can go in any one of the $i$ pots. There are $\binom{n-1}{\ell-1}$ ways to choose the other $\ell-1$ vertices of the tree, and if we highlight exactly $j$ consecutive instances of $\sigma$ in this tree, there are then $T(\ell, j)$ possibilities for the tree. So, if $n\ge 1$,
\begin{equation}\label{eq:F(n;i;k)_recurrence}
F(n, i, m)=\sum_{\ell=1}^n i\binom{n-1}{\ell-1}\sum_{j=0}^m T(\ell, j)F(n-\ell, i, m-j).
\end{equation}

To calculate $T(n, m)$, again for $n\ge 1$, we do casework on whether the root vertex is part of a highlighted consecutive instance of $\sigma$. If it is, then it is contained in a uniquely determined forest cluster. If this forest cluster consists of $\ell$ vertices and has exactly $j$ highlighted consecutive instances, the cluster has $\binom{n}{\ell}r_{\ell, j}$ possibilities. Then, we can think of adding the rest of the vertices as creating a forest with $\ell$ possible pots on $n-\ell$ vertices and with $m-j$ highlighted consecutive instances of $\sigma$. Otherwise, if the root is not in a highlighted consecutive instance of $\sigma$, we reduce to a forest on $n-1$ vertices with $m$ highlighted consecutive instances of $\sigma$. So, for $n\ge 1$,
\begin{equation}\label{eq:T(n;k)_recurrence}
T(n, m)=nF(n-1, 1, m)+\sum_{\ell=1}^n\sum_{j=1}^m \binom{n}{\ell}r_{\ell, j}F(n-\ell, \ell, m-j).
\end{equation}

We first show that the forest cluster numbers determine the values $F(n,1,m)$. Regardless of the pattern $\sigma$, the quantity $F(0,i,m)$ equals $1$ if $m=0$, and $0$ if $m>0$. Inducting on $n$ and using \cref{eq:F(n;i;k)_recurrence} and \cref{eq:T(n;k)_recurrence}, we find the forest cluster numbers uniquely determine the values $F(n,i,m)$, and in particular the values $F(n,1,m)$, in a manner independent of $\sigma$, as claimed.

For the other direction, we must show that the values $F(n,1,m)$ determine the forest cluster numbers. The values $F(n,1,m)$ first determine all values of the form $F(n,i,m)$, as
\[F(n,i,m)=\sum_{\substack{n_1+\cdots+n_i=n\\n_1,\dots,n_i\in\mathbb{Z}_{\ge 0}}}\binom{n}{n_1,\dots,n_i}\sum_{\substack{m_1+\cdots+m_i=m\\m_1,\dots,m_i\in\mathbb{Z}_{\ge 0}}}\prod_{j=1}^iF(n_j,1,m_j).\]
This equation follows by considering all possible ways to assign the $n$ vertices and the $m$ consecutive instances to the $i$ pots. Also, by \cref{eq:F(n;i;k)_recurrence}, for all $n\ge 1$ and $m\ge 0$,
\[T(n,m)=F(n,1,m)-\sum_{\substack{1\le \ell\le n\\0\le j\le m\\(\ell,j)\neq(n,m)}}\binom{n-1}{\ell-1}F(n-\ell,1,m-j)T(\ell,j).\]
Thus inductively, we also determine all $T(n,m)$ from the values $F(n,i,m)$, starting with $T(1,0)=F(1,1,0)$. Finally, by \cref{eq:T(n;k)_recurrence}, for all $n,m\ge 1$,
\[r_{n,m}=T(n,m)-nF(n-1, 1, m)-\sum_{\substack{1\le \ell\le n\\1\le j\le m\\(\ell,j)\neq(n,m)}}\binom{n}{\ell}r_{\ell, j}F(n-\ell, \ell, m-j).\]
Thus the forest cluster numbers are inductively determined from the values $T(n,m)$ and $F(n,i,m)$. Moreover, as was the case in the reverse direction, the forced values of the forest cluster numbers can be computed without knowledge of $\sigma$, so we are done.
\end{proof}

Determining strong-c-forest-Wilf equivalences is therefore equivalent to determining whether the forest cluster numbers are equal. A straightforward computation yields a formula for $r_{2k-1,2}$ for an arbitrary pattern $\sigma\in\mathcal{S}_k$, and this single forest cluster number allows us to prove \cref{thm:same-first-number} below.

\begin{proposition}\label{prop:cluster-number}
Let $\sigma=\sigma(1)\cdots\sigma(k)\in\mathcal{S}_k$. Then,
\[r_{2k-1,2}=\binom{2k-1}{k}-\frac{1}{2}\binom{2\sigma(1)-2}{\sigma(1)-1}\binom{2k-2\sigma(1)}{k-\sigma(1)}.\]
\end{proposition}

\begin{proof}
We first count the number of ordered pairs $(A,B)$ where $A$ and $B$ are sequences of $k$ distinct positive integers taken from the set $\{1,2,\dots,2k-1\}$ such that $A$ and $B$ are both in the same relative order as $\sigma$. We impose the condition that $A$ and $B$ have exactly one number in common, which must be the first number of $B$. So, each element of the set $\{1,2,\dots,2k-1\}$ appears in $A$ or $B$. The number of such ordered pairs is easily seen to be $\binom{2k-1}{k}$; after choosing the $k$ numbers that appear in $B$, the order of $B$ is uniquely determined, and there is then exactly one possibility for $A$.

The number of forest clusters counted by $r_{2k-1,2}$ is almost exactly given by the number of ordered pairs $(A,B)$ counted above, where the label of the root of the cluster corresponds to the first number of $A$, and $A$ and $B$ correspond to the highlighted consecutive instances. However, there is overcounting if the starting point of each highlighted consecutive instance is the root; in this case, we obtain the same forest if we swap $A$ and $B$. There are $\binom{2\sigma(1)-2}{\sigma(1)-1}\binom{2k-2\sigma(1)}{k-\sigma(1)}$ ordered pairs $(A,B)$ where $A$ and $B$ start with the same number, so we must subtract half of this amount from $\binom{2k-1}{k}$, leading to the desired expression for $r_{2k-1,2}$.
\end{proof}

We can now give a proof of \cref{thm:same-first-number}.
\begin{proof}[Proof of \cref{thm:same-first-number}]
If two patterns $\sigma, \tau \in\mathcal{S}_k$ are strong-c-forest-Wilf equivalent, then by \cref{thm:same-cluster-numbers}, all their forest cluster numbers must be equal. In particular, their values for $r_{2k-1, 2}$ are equal, so by \cref{prop:cluster-number},
\[\binom{2\sigma(1)-2}{\sigma(1)-1}\binom{2k-2\sigma(1)}{k-\sigma(1)}=\binom{2\tau(1)-2}{\tau(1)-1}\binom{2k-2\tau(1)}{k-\tau(1)}.\]
This implies $\{\sigma(1)-1, k-\sigma(1)\}=\{\tau(1)-1, k-\tau(1)\}$,
giving the result.
\end{proof}

\subsection{$1423$ and $1324$ are c-forest-Wilf equivalent}\label{subsec:1423-1324}

We will now show that the forest cluster numbers are the same for $1324$ and $1423$. Throughout this section, we restrict our attention to length-$4$ patterns $\sigma=\sigma(1)\sigma(2)\sigma(3)\sigma(4)$ such that $\sigma(1)<\sigma(2)$ and $\sigma(3)<\sigma(4)$. In \cref{subsubsec:involution-on-proper-twig-collections}, we introduce objects we refer to as \textit{$\sigma$-(extra)nice trees} and \textit{proper twig collections}, and then construct an involution on proper twig collections. In \cref{subsubsec:forest-clusters-nice-trees}, after using the involution to relate the number of $1423$-nice trees to the number of $1324$-nice trees, we relate forest clusters with respect to $\sigma$ to $\sigma$-nice trees, allowing us to prove \cref{thm:1324-1423}. Finally, in \cref{subsubsec:sigma-extranice-enumeration}, we show that the numbers of $\sigma$-extranice trees on $[n]$ are equal for $\sigma=1234,1423,1324$, and then provide an explicit formula for their quantity.

\subsubsection{An involution on proper twig collections}\label{subsubsec:involution-on-proper-twig-collections}

We first define a certain type of labeled forest with respect to $\sigma$. Recall from \cref{sec:preliminaries} the notions of the \textit{depth} of a vertex and the \textit{depth} of a tree, and that the root of a tree has depth $1$.

\begin{definition}\label{def:(extra)nice}
A labeled rooted forest (resp.\ tree) with distinct positive integer labels is \textit{$\sigma$-nice} if the following conditions are satisfied:
\begin{itemize}
    \item Every vertex of odd depth has at least one child.
    \item Every vertex of even depth has a label greater than that of its parent, and every vertex of even depth with depth at least $4$ is
    the endpoint of a consecutive instance of $\sigma$.
\end{itemize}
A forest (resp.\ tree) is \textit{$\sigma$-extranice} if it is $\sigma$-nice and the following condition holds:
\begin{itemize}
\item Every vertex of odd depth has exactly one child.
\end{itemize}
\end{definition}

Note that by the convention mentioned in \cref{sec:preliminaries}, a $\sigma$-(extra)nice tree must have at least one vertex (and therefore at least two vertices, by the first condition). Also, the definition only requires the labels of a $\sigma$-(extra)nice forest or tree to be distinct positive integers, so the labels do not necessarily have to form the set $[n]$ for some $n$. This is done for convenience later in this section.

It follows from \cref{def:(extra)nice} that a $\sigma$-extranice tree or forest can only exist on an even number of vertices. For the purposes of proving that $1423$ and $1324$ are strong-c-forest-Wilf equivalent, we only need to work with $1423$-nice trees and $1324$-nice trees, but we will have more to say about $\sigma$-extranice forests and $\sigma$-extranice trees in \cref{subsubsec:sigma-extranice-enumeration}. There are exactly four $1423$-extranice trees on $\{1,2,\dots,6\}$ and four $1324$-extranice trees on $\{1,2,\dots,6\}$, which are shown in \cref{fig:extranice1,fig:extranice2}, respectively.

\begin{figure}
\centering
\includegraphics{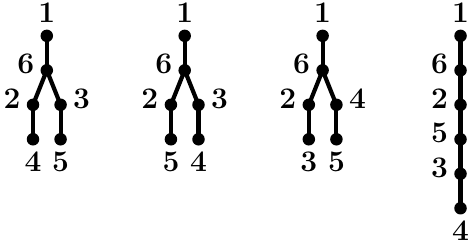}
\caption{The four $1423$-extranice trees on $\{1, 2, 3, 4, 5, 6\}$.}
\label{fig:extranice1}
\end{figure}

\begin{figure}
\centering
\includegraphics{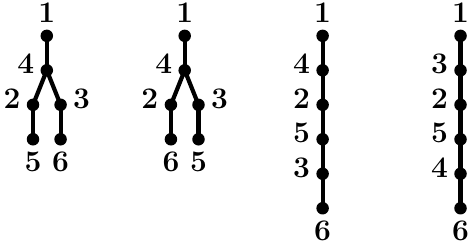}
\caption{The four $1324$-extranice trees on $\{1, 2, 3, 4, 5, 6\}$.}
\label{fig:extranice2}
\end{figure}

In a $\sigma$-nice forest or tree, by definition every vertex of odd depth has a nonempty set of children. Motivated by this fact, we make the following definition:

\begin{definition}
A \textit{twig} is a labeled rooted tree of depth exactly $2$ consisting of a \textit{parent vertex} (the root) along with at least one adjacent \textit{child vertex}, such that the labels of the child vertices are distinct. A twig is \textit{proper} if the label of each child vertex is greater than the label of the parent vertex (which in particular implies that all vertices have distinct labels).

The set of labels of the child vertices of a twig $t$ is the \textit{child label set} of $t$, and the label of the parent vertex of $t$ is the \textit{parent label} of $t$. We denote a twig with parent label $p$ and child label set $C$ as $(p, C)$.
\end{definition}

While we only work with twigs that have distinct vertex labels, we use a weaker definition of twig as we a priori do not know about the distinctness of labels under the $\gamma$ map below (a result proved in \cref{lem:gamma-child-vertex-labels-permutation}).

A $\sigma$-nice forest or tree has a uniquely determined \textit{decomposition} into proper twigs, a concept we formally define below. Before that, we must define certain sets of twigs.

\begin{definition}
A \textit{twig collection} is a nonempty set of twigs in which the child label sets of the twigs are disjoint. A twig collection is \textit{proper} if each of the twigs is proper and all the vertex labels among the twigs are distinct.

A \textit{parent vertex} (resp.\ \textit{child vertex}) of a twig collection $W$ is a parent vertex (resp.\ child vertex) of a twig in $W$.

The \textit{child label set} of a twig collection $W$ is the set of labels of the child vertices of $W$ (i.e., the union of the child label sets of the twigs of $W$). The \textit{label set} of a proper twig collection $W$ is the set of labels of the vertices of $W$.
\end{definition}
Note that a twig collection consisting only of proper twigs is not necessarily a proper twig collection, though this distinction will be irrelevant once we prove \cref{lem:gamma-child-vertex-labels-permutation}.

\begin{definition}\label{def:nice-decomposition}
A $\sigma$-nice forest $F$ (resp.\ tree $T$) has a \textit{decomposition} into a proper twig collection $W$, where every odd-depth vertex $v$ of $F$ (resp.\ in $T$), along with its children, becomes a twig in $W$. We then say that $F$ (resp.\ $T$) is \textit{constructed from $W$}. Given a proper twig collection $W$, let $F_\sigma(W)$ (resp.\ $T_\sigma(W)$) be the number of $\sigma$-nice forests (resp.\ trees) constructed from $W$.
\end{definition}
The fact that $W$ is a proper twig collection follows from the definition of a $\sigma$-nice forest.

\begin{figure}
    \centering
    \includegraphics{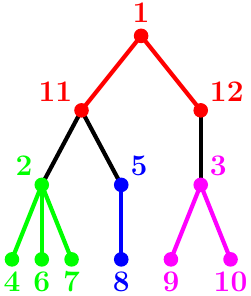}
    \caption{The decomposition of a $1423$-nice tree into a proper twig collection, namely $\{(1, \{11, 12\}), (2, \{4, 6, 7\}), (5, \{8\}), (3, \{9, 10\})\}$.}
    \label{fig:twigdecomp}
\end{figure}

\cref{fig:twigdecomp} shows the decomposition of a $1423$-nice tree into a proper twig collection with $4$ twigs. We will use the decompositions of $\sigma$-nice trees to relate certain counts of $1423$-nice trees on $[n]$ to the corresponding counts of $1324$-nice trees on $[n]$; for instance, one of our results will imply that the number of $1423$-nice trees on $[n]$ equals the number of $1324$-nice trees on $[n]$ (a result that is later generalized in \cref{thm:extranice-equal}). Then, to prove \cref{thm:1324-1423}, we will use these relations to show that the forest cluster numbers of $1423$ and the forest cluster numbers of $1324$ satisfy the same recurrence.

\begin{definition}\label{def:reltwig}
Let $W$ be a twig collection and $T$ be a set of positive integers such that the total number of child vertices of $W$ is $|T|$. Then, we define $\operatorname{rel}_T(W)$ to be the twig collection obtained by relabeling the vertices of $W$ so that the label of each parent vertex is preserved, the child label set becomes $T$, and the initial relative order of the labels of the child vertices across all twigs in $W$ is preserved.
\end{definition}
We give an example to illustrate \cref{def:reltwig}.
\begin{example}
Let $W=\{(1, \{2, 4\}), (3, \{5, 8\})\}$ be a twig collection with parent vertices labeled $1$ and $3$, and let $T=\{5, 7, 9, 11\}$. Then, $\operatorname{rel}_T(W)=\{(1, \{5, 7\}), (3, \{9, 11\})\}$.
\end{example}

\begin{definition}
Let $t$ be a twig and let $E=\{x_1,x_2,\dots,x_\ell\}$ be a finite set of positive integers that contains the child label set of $t$ as a subset, where $x_1<x_2<\cdots<x_\ell$. We define the twig $\alpha_E(t)$ to be a relabeled version of $t$ in which the parent label remains the same, but for each child vertex $v$, if in $t$ vertex $v$ is labeled $x_i$, then in $\alpha_E(t)$ vertex $v$ is labeled $x_{\ell+1-i}$.
\end{definition}

\begin{definition}\label{def:gamma-operation}
We recursively define a map $\gamma\colon \{\text{proper twig collections}\} \to \{\text{twig collections}\}$. Let $W=\{t_1, t_2, \dots, t_s\}$ be a proper twig collection, where the twigs $t_1,\dots,t_s$ are ordered in increasing order of parent label. If $s=1$, then we define $\gamma(W)=W$. Otherwise, let $C$ be the child label set of $W$. Furthermore, let $E$ be the subset of $C$ that consists of labels that are larger than the parent label of $t_s$, and let $D$ be the child label set of $\alpha_E(t_s)$. Then, we define
\[\gamma(W)=\operatorname{rel}_{C\setminus D}(\gamma(W \setminus \{t_s\})) \cup \{\alpha_E(t_s)\}.\]
\end{definition}

By inducting on $s$, we can see that the recursive process terminates, each expression appearing in the definition is defined, $\gamma(W)$ is a twig collection, and $\gamma(W)$ differs from $W$ only by a relabeling of the child vertices, as $\alpha_E$ and $\operatorname{rel}_{C\setminus D}$ do not change the parent label of any twig or the number of children any parent vertex has.

\begin{example}
Say we have the proper twig collection $W=\{(1, \{3, 4\}), (2, \{5, 7\}), (6, \{8\})\}$. When applying $\gamma$, we have $C=\{3, 4, 5, 7, 8\}$ and $E=\{7, 8\}$. Then, $\alpha_{\{7, 8\}}((6, \{8\})=(6, \{7\})$, so
\[\gamma(W)=\operatorname{rel}_{\{3,4,5,8\}}(\gamma(\{(1,\{3,4\}),(2,\{5,7\})\}))\cup\{(6,\{7\})\}.\]

Now, when applying $\gamma$ to $\{(1, \{3, 4\}), (2, \{5, 7\})\}$, the four vertex labels greater than $2$ are $3$, $4$, $5$, and $7$. Since $2$ is currently connected to the greatest two of the four vertices, applying $\gamma$ makes it connect to the least two of the four vertices, namely $3$ and $4$. Thus
\[\alpha_{\{3, 4, 5, 7\}}((2, \{5, 7\}))=(2, \{3, 4\}).\]
So, the second twig in $\{(1, \{3, 4\}), (2, \{5, 7\})\}$ changes to $(2, \{3, 4\})$. Therefore,
\[\gamma(\{(1, \{3, 4\}), (2, \{5, 7\})\})=\{(1, \{5, 7\}), (2, \{3, 4\})\},\]
because the labels of the child vertices of $(1, \{3, 4\})$ are relabeled to be the remaining labels in $\{3, 4, 5, 7\}$:
\[\operatorname{rel}_{\{5, 7\}}(\{(1, \{3, 4\})\})=\{(1, \{5, 7\})\}.\]
Under the $\operatorname{rel}$ operation, the child vertices of these two twigs are relabeled using the set $\{3, 4, 5, 7, 8\} \setminus \{7\}$:
\[\operatorname{rel}_{\{3, 4, 5, 8\}}(\{(1, \{5, 7\}), (2, \{3, 4\})\})=\{(1, \{5, 8\}), (2, \{3, 4\})\}.\]
We end up with $\gamma(W)=\{(1, \{5, 8\}), (2, \{3, 4\}), (6, \{7\})\}$.
\end{example}

\begin{lemma}\label{lem:gamma-child-vertex-labels-permutation}
Let $W$ be a proper twig collection. Then, the twig collection $\gamma(W)$ differs from $W$ only by a permutation of the labels of the child vertices, and consequently all the labels of the vertices of $\gamma(W)$ are distinct.
\end{lemma}
\begin{proof}
As noted after \cref{def:gamma-operation}, the labels of the child vertices of $\gamma(W)$ are distinct, and the twig collection $\gamma(W)$ differs from $W$ only by a relabeling of the child vertices. If $W$ consists of exactly one twig, then $\gamma(W)=W$, and we are done. Otherwise, using the notation in \cref{def:gamma-operation}, the child label set of $\gamma(W)$ is $(C\setminus D)\cup D=C$, and again we are done.
\end{proof}

\begin{lemma}\label{lem:specialrel}
Let $W=\{t_1,\dots,t_s\}$ be a proper twig collection, where $s\ge 2$. Using the notation in \cref{def:gamma-operation}, where $\gamma(W)=\operatorname{rel}_{C\setminus D}(\gamma(W \setminus \{t_s\})) \cup \{\alpha_E(t_s)\}$, each of the twig collections $\gamma(W \setminus \{t_s\})$ and $\operatorname{rel}_{C\setminus D}(\gamma(W \setminus \{t_s\}))$ has all its vertex labels distinct, and furthermore the two twig collections have all their vertex labels in the same relative order.
\end{lemma}
\begin{proof}
By \cref{lem:gamma-child-vertex-labels-permutation}, the twig collection $\gamma(W\setminus\{t_s\})$ has all its vertex labels distinct. Similarly, the same is true for $\gamma(W)$, and therefore also for $\operatorname{rel}_{C\setminus D}(\gamma(W \setminus \{t_s\}))\subseteq\gamma(W)$. We now show the vertex labels of $\gamma(W\setminus\{t_s\})$ and $\operatorname{rel}_{C\setminus D}(\gamma(W \setminus \{t_s\}))$ are in the same relative order.

When applied to any twig collection, the $\operatorname{rel}$ operation preserves the label of each parent vertex and preserves the relative order of the labels of the child vertices. Let $D_0$ be the child label set of the twig $t_s$, so that $D_0,D\subseteq E$. By \cref{lem:gamma-child-vertex-labels-permutation}, the child label set of $\gamma(W\setminus\{t_s\})$ is $C\setminus D_0$. Applying $\operatorname{rel}_{C\setminus D}$, we change the child label set of $\gamma(W\setminus\{t_s\})$ from $C\setminus D_0$ to $C\setminus D$, which does not change the label of any child vertex whose label lies in $C\setminus E$. Therefore, if a child vertex $v$ of $\gamma(W\setminus\{t_s\})$ has a label $\ell$ that is less than the label of some parent vertex of $\gamma(W\setminus\{t_s\})$, then $\ell$ must lie in $C\setminus E$, and so the label of $v$ is unchanged. So, when applying $\operatorname{rel}_{C \setminus D}$, a child vertex only changes labels if its original label is already greater than the label of each parent vertex of $\gamma(W\setminus\{t_s\})$, and it changes to another label that is greater than the label of each parent vertex of $\gamma(W\setminus\{t_s\})$. Thus, all the vertex labels of $\gamma(W\setminus \{t_s\})$ and $\operatorname{rel}_{C \setminus D}(\gamma(W \setminus \{t_s\}))$ are in the same relative order.
\end{proof}

\begin{lemma}\label{lem:gamma(W)-proper}
The twig collection $\gamma(W)$ is proper for any proper twig collection $W$.
\end{lemma}

\begin{proof}
Using the notation of \cref{def:gamma-operation}, we write $W=\{t_1,\dots,t_s\}$, where $s\ge 1$. We induct on $s$, with the base case $s=1$ being trivial, since in this case $\gamma(W)=W$. Now assume $s\ge 2$, and suppose the statement is true for all proper twig collections with $s-1$ twigs. We know that $W\setminus\{t_s\}$ is a proper twig collection, and by the inductive hypothesis, so is $\gamma(W\setminus\{t_s\})$. In particular, the label of each child vertex of $\gamma(W\setminus\{t_s\})$ is greater than the label of its parent. By \cref{lem:specialrel}, the same is true for $\operatorname{rel}_{C\setminus D}(\gamma(W\setminus\{t_s\}))$. This is also true for the twig collection $\{\alpha_E(t_s)\}$. Since all the vertex labels of $\gamma(W)$ are distinct by \cref{lem:gamma-child-vertex-labels-permutation}, we see that $\gamma(W)=\operatorname{rel}_{C\setminus D}(\gamma(W\setminus\{t_s\}))\cup\{\alpha_E(t_s)\}$ must be a proper twig collection.
\end{proof}

\begin{proposition}\label{prop:gamma-involution}
We have $\gamma(\gamma(W))=W$ for all proper twig collections $W$. Thus, $\gamma$ induces an involution $\{\emph{proper twig collections}\} \to \{\emph{proper twig collections}\}$.
\end{proposition}
\begin{remark}
Note that $\gamma(\gamma(W))$ is well-defined by \cref{lem:gamma(W)-proper}.
\end{remark}
\begin{proof}
We prove that $\gamma(\gamma(W))=W$ by inducting on the number of twigs of $W$. If $W$ consists of $1$ twig, then clearly $\gamma(\gamma(W))=W$. Otherwise, suppose $W$ has $s\ge 2$ twigs, and suppose $\gamma(\gamma(W'))=W'$ for all proper twig collections $W'$ with $s-1$ twigs. As in \cref{def:gamma-operation}, let $W=\{t_1, t_2, \dots, t_s\}$, where the twigs $t_1,\dots,t_s$ are ordered in increasing order of parent label. Note that the twig in $\gamma(W)$ with the largest parent label is $\alpha_E(t_s)$. Let $C$ be the child label set of $W$, which is also the child label set of $\gamma(W)$ by \cref{lem:gamma-child-vertex-labels-permutation}. Furthermore, let $E$ be the subset of $C$ consisting of labels that are larger than the parent label of $t_s$. Note the parent label of $t_s$ equals the parent label of $\alpha_E(t_s)$. Let $D_0$ be the child label set of $t_s$, and let $D$ be the child label set of $\alpha_E(t_s)$.

It is clear that $\alpha_E(\alpha_E(t_s))=t_s$, so the child label set of $\alpha_E(\alpha_E(t_s))$ is $D_0$. We have
\begin{align*}
    \gamma(\gamma(W))&=\gamma(\operatorname{rel}_{C\setminus D}(\gamma(W\setminus\{t_s\}))\cup\{\alpha_E(t_s)\})\\
    &=\operatorname{rel}_{C\setminus D_0}(\gamma(\operatorname{rel}_{C\setminus D}(\gamma(W\setminus\{t_s\}))))\cup\{\alpha_E(\alpha_E(t_s))\}\\
    &=\operatorname{rel}_{C\setminus D_0}(\gamma(\operatorname{rel}_{C\setminus D}(\gamma(W\setminus \{t_s\})))) \cup \{t_s\}.
\end{align*}

We claim that $\gamma(\operatorname{rel}_{C\setminus D}(\gamma(W\setminus \{t_s\})))=\operatorname{rel}_{C\setminus D}(W\setminus \{t_s\})$. To see this, recall that $\operatorname{rel}_{C\setminus D}$, when applied to $\gamma(W \setminus \{t_s\})$, keeps all vertex labels in the same relative order by \cref{lem:specialrel}. Both $\gamma(W \setminus \{t_s\})$ and $\operatorname{rel}_{C\setminus D}(\gamma(W \setminus \{t_s\}))$ are proper, so we can apply $\gamma$ to each one, and we see that $\gamma(\gamma(W\setminus \{t_s\}))=W\setminus\{t_s\}$ and $\gamma(\operatorname{rel}_{C\setminus D}(\gamma(W\setminus \{t_s\})))$ also have all their vertex labels in the same relative order (and each one has all its vertex labels distinct by \cref{lem:gamma-child-vertex-labels-permutation}). The equality $\gamma(\gamma(W\setminus\{t_s\}))=W\setminus\{t_s\}$ follows from the inductive hypothesis. By an argument similar to the one used to prove \cref{lem:specialrel}, the twig collections $W\setminus\{t_s\}$ and $\operatorname{rel}_{C\setminus D}(W\setminus\{t_s\})$ each have all their labels distinct, and they have all their vertex labels in the same relative order. Therefore, $\gamma(\operatorname{rel}_{C\setminus D}(\gamma(W\setminus \{t_s\})))$ and $\operatorname{rel}_{C\setminus D}(W\setminus\{t_s\})$ have all their vertex labels in the same relative order. Since these two twig collections also have the same set of all vertex labels, they are equal.

We then arrive at 
\begin{align*}
    \gamma(\gamma(W))&=\operatorname{rel}_{C\setminus D_0}(\gamma(\operatorname{rel}_{C\setminus D}(\gamma(W\setminus \{t_s\})))) \cup \{t_s\}\\
    &=\operatorname{rel}_{C\setminus D_0}(\operatorname{rel}_{C\setminus D}(W\setminus \{t_s\})) \cup \{t_s\}\\
    &=(W\setminus \{t_s\}) \cup \{t_s\}\\
    &=W,
\end{align*}
and therefore $\gamma$ induces an involution on the set of proper twig collections.
\end{proof}

A key result we will use to prove that $1423$ and $1324$ are strong-c-forest-Wilf equivalent is the following:

\begin{proposition}\label{prop:twig}
Let $W$ be a proper twig collection. Then, $F_{1423}(W)=F_{1324}(\gamma(W))$ and $T_{1423}(W)=T_{1324}(\gamma(W))$.
\end{proposition}

\begin{proof}
Let $W$ have $n$ total vertices and $s$ twigs. We induct on $s$. If $s=1$, the result is clear, as $F_{1423}(W)=T_{1423}(W)=F_{1324}(\gamma(W))=T_{1324}(\gamma(W))=1$. So, suppose $W$ has $s\ge 2$ twigs, and suppose the statement is true for proper twig collections containing exactly $s-1$ twigs.

Without loss of generality, suppose that $W$ has label set $[n]$. We denote the twigs of $W$ as $t_1, t_2, \dots, t_s$, ordered in increasing order of parent labels, and let $C$ be the child label set of $W$. Let $p$ be the parent label of $t_s$, and let $c$ and $d$ be the smallest and largest labels of the child label set of $t_s$, respectively. Let $E=\{p+1,p+2,\dots,n\}$ be the subset of the child label set of $W$ that consists of labels that are larger than $p$, so that the twig of $\gamma(W)$ with the largest parent label is $\alpha_E(t_s)$. Note that the smallest and largest members of the child label set of $\alpha_E(t_s)$ are $n+p+1-d$ and $n+p+1-c$, respectively. Let $D$ be the child label set of the twig $\alpha_E(t_s)$.

We will derive the relations
\begin{align}
    F_{1423}(W)&=(n-d+1)F_{1423}(W\setminus\{t_s\}),\label{eq:1423-forests}\\
    T_{1423}(W)&=(n-d)T_{1423}(W\setminus\{t_s\}),\label{eq:1423-trees}\\
    F_{1324}(W)&=(c-p)F_{1324}(W\setminus\{t_s\}),\label{eq:1324-forests}\\
    T_{1324}(W)&=(c-p-1)T_{1324}(W\setminus\{t_s\}).\label{eq:1324-trees}
\end{align}
Given these relations, plugging in $\gamma(W)$ for $W$ in \cref{eq:1324-forests,eq:1324-trees}, and noting from \cref{lem:specialrel} that the proper twig collections $\gamma(W \setminus \{t_s\})$ and $\operatorname{rel}_{C \setminus D}(\gamma(W \setminus \{t_s\}))$ have all their vertex labels in the same relative order, we find
\begin{align*}
F_{1324}(\gamma(W))&=(n-d+1)F_{1324}(\operatorname{rel}_{C\setminus D}(\gamma(W\setminus\{t_s\})))=(n-d+1)F_{1324}(\gamma(W \setminus \{t_s\})),\\
T_{1324}(\gamma(W))&=(n-d)T_{1324}(\operatorname{rel}_{C\setminus D}(\gamma(W\setminus\{t_s\})))=(n-d)T_{1324}(\gamma(W \setminus \{t_s\})).
\end{align*}
After applying the inductive hypothesis, the result then follows by comparing these two equations with \cref{eq:1423-forests,eq:1423-trees}. It remains to show \cref{eq:1423-forests,eq:1423-trees,eq:1324-forests,eq:1324-trees}.

Consider a $1423$-nice forest (resp.\ tree) constructed from $W$. There cannot be a twig of $W$ whose parent vertex is a child of a child vertex of $t_s$, as this would violate the $1423$-nice condition. Therefore, removing twig $t_s$ from our $1423$-nice forest (resp.\ tree) gives a $1423$-nice forest (resp.\ tree) constructed from $W\setminus\{t_s\}$. Conversely, suppose we are given a $1423$-nice forest constructed from $W\setminus\{t_s\}$. To add twig $t_s$ to form a $1423$-nice forest constructed from $W$, we must choose some child vertex of $W\setminus\{t_s\}$ to be the parent of the parent vertex of $t_s$, or decide that twig $t_s$ forms its own tree. We may choose any child vertex if and only if its label is greater than $d$, and conversely each of the labels $d+1,d+2,\dots,n$ is the label of some child vertex of $W\setminus\{t_s\}$. Thus, there are exactly $n-d+1$ ways to add $t_s$ to the $1423$-nice forest constructed from $W\setminus\{t_s\}$ to form a $1423$-nice forest constructed from $W$. This proves \cref{eq:1423-forests}.

The proof of \cref{eq:1423-trees} is similar; everything is the same as in the proof of \cref{eq:1423-forests}, except that the twig $t_s$ cannot form its own tree, giving $n-d$ ways instead of $n-d+1$ ways.

Now consider a $1324$-nice forest (resp.\ tree) constructed from $W$. Twig $t_s$ has the largest parent label among the twigs in $W$, so there cannot be a twig with its parent vertex attached to a child vertex of $t_s$, and removing $t_s$ gives a $1324$-nice forest (resp.\ tree) constructed from $W\setminus\{t_s\}$. Conversely, given a $1324$-nice forest constructed from $W\setminus\{t_s\}$, to add $t_s$ to form a $1324$-nice forest constructed from $W$, we must choose some child vertex of $W\setminus\{t_s\}$ to be the parent of the parent vertex of $t_s$, or decide that twig $t_s$ forms its own tree. Such a child vertex gives us a valid $1324$-nice forest if and only if its label is greater than $p$ but less than $c$, and each of the labels $p+1,p+2,\dots,c-1$ appears as the label of some child vertex of $W\setminus\{t_s\}$. There are $c-p-1$ such labels, so there are $c-p$ total ways to add $t_s$. This holds for any $1324$-nice forest constructed from $W\setminus\{t_s\}$, so \cref{eq:1324-forests} holds. In the case of $1324$-nice trees, using similar reasoning, there are only $c-p-1$ ways to add $t_s$, and we obtain \cref{eq:1324-trees}.
\end{proof}

\subsubsection{Creating forest clusters from $\sigma$-nice trees}\label{subsubsec:forest-clusters-nice-trees}
We now relate the counts of $1423$-nice trees and $1324$-nice trees. We will then use these counts to count forest clusters with respect to $1423$ and $1324$.

\begin{definition}
For nonnegative integers $n$ and $m$, let $A_\sigma(n, m)$ be the number of $\sigma$-nice trees on $[n]$ containing exactly $m$ consecutive instances of $\sigma$.
\end{definition}
\begin{proposition}\label{prop:A1423=A1324}
For all $n,m\ge 0$, we have $A_{1423}(n,m)=A_{1324}(n,m)$.
\end{proposition}
\begin{proof}
Fix $n, m$. Let the pattern $\sigma$ be either $1423$ or $1324$, and let $T$ be a $\sigma$-nice tree on $[n]$. Suppose $T$ decomposes into the proper twig collection $W$, as defined in \cref{def:nice-decomposition}. Let $W=\{t_1,\dots,t_s\}$, where the twigs are ordered in increasing order of parent label. Note that the tree $T$ contains exactly $m$ consecutive instances of $\sigma$ if and only if there are exactly $m$ child vertices of $W$ that are not vertices of $t_1$. This follows from the observation that a vertex of $T$ is the endpoint of a consecutive instance of $\sigma$ if and only if it corresponds to a child vertex of a twig $t_j$ of $W$ such that $j>1$, as the root vertex of $T$ necessarily corresponds to the parent vertex of $t_1$. Therefore, $A_\sigma(n,m)$ equals the sum of $T_\sigma(W)$ for all proper twig collections $W$ such that $W$ has label set $[n]$ and there are exactly $m$ child vertices of $W$ that are not vertices of $t_1$, the twig of $W$ with the smallest parent label. By \cref{lem:gamma-child-vertex-labels-permutation}, the proper twig collection $W$ satisfies this property if and only if $\gamma(W)$ does. So, the result follows from \cref{prop:gamma-involution,prop:twig}.
\end{proof}

We now make the following definition:
\begin{definition}
For nonnegative integers $n$ and $m$, let $B_\sigma(n, m)$ be the number of $\sigma$-nice trees on $[n]$ containing exactly $m$ consecutive instances of $\sigma$, such that there are no childless depth-$2$ vertices.
\end{definition}

\begin{proposition}\label{prop:B1423=B1324}
For all $n,m\ge 0$, we have $B_{1423}(n, m)=B_{1324}(n, m)$.
\end{proposition}
\begin{proof}
Let the pattern $\sigma$ be either $1423$ or $1324$. Given $L\subseteq [n]$, let $A_\sigma(n,m,L)$ equal the number of $\sigma$-nice trees $T$ on $[n]$ containing exactly $m$ consecutive instances of $\sigma$ such that $L$ equals the set of labels of the childless depth-$2$ vertices. If $n=1$ or $1\in L$, then trivially $A_\sigma(n,m,L)=0$, because any $\sigma$-nice tree has at least two vertices and has its root vertex labeled $1$. So, suppose $n\ge 2$ and $1\notin L$. If we remove the childless depth-$2$ vertices from such a tree, the remaining rooted tree either
\begin{enumerate}[label=(\arabic*)]
    \item only consists of the root vertex (which occurs exactly once if and only if $m=0$ and $L=\{2,\dots,n\}$, and otherwise does not occur), or
    \item forms a $\sigma$-nice tree on $[n]\setminus L$, with no childless depth-$2$ vertices, and containing exactly $m$ consecutive instances of $\sigma$.
\end{enumerate}
Conversely, given any $\sigma$-nice tree satisfying the conditions given in (2), we see that adding (childless) children to the root vertex and labeling them with the set $L$ yields a $\sigma$-nice tree on $[n]$ that is counted by $A_\sigma(n,m,L)$. So, for all $n,m$ and $L\subseteq[n]$, we find
\[A_\sigma(n,m,L)=\mathbf{1}_{n\ge 2}\mathbf{1}_{1\notin L}\left(\mathbf{1}_{m=0}\mathbf{1}_{L=\{2,\dots,n\}}+B_\sigma(n-|L|, m)\right).\]
Thus we obtain the recurrence
\[A_\sigma(n, m)=\sum_{L\subseteq[n]}A_\sigma(n, m, L)=\mathbf{1}_{n\ge 2}\sum_{L\subseteq\{2,\dots,n\}}\left(\mathbf{1}_{m=0}\mathbf{1}_{L=\{2,\dots,n\}}+B_\sigma(n-|L|, m)\right).\]
Now fix $m\ge 0$. We induct on $n$ to obtain the desired result. Note that if $n=0,1$, then $B_{1423}(n,m)=B_{1324}(n,m)=0$. For $n\ge 2$, we have
\[A_\sigma(n,m)=\mathbf{1}_{m=0}+\sum_{i=0}^{n-1}\binom{n-1}{i}B_\sigma(n-i,m),\]
so $B_\sigma(n,m)=A_\sigma(n,m)-\mathbf{1}_{m=0}-\sum_{i=1}^{n-1}\binom{n-1}{i}B_\sigma(n-i,m)$. Therefore, $B_\sigma(n, m)$ is uniquely determined from $A_\sigma(n, m)$ and the values $B_\sigma(n', m)$ for $n'<n$. The desired result then follows from \cref{prop:A1423=A1324} and the fact that the recurrence and initial values are the same for $\sigma=1423$ and $\sigma=1324$.
\end{proof}

\begin{figure}
\centering
    \includegraphics{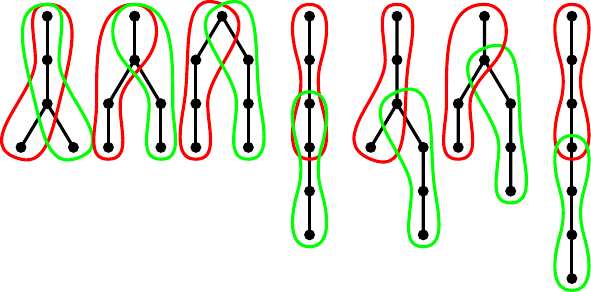}
    \caption{The seven ways for two consecutive instances of $\sigma=1423$ to overlap, which are the same for $\sigma=1324$.}
    \label{fig:sigma-overlap}
\end{figure}

We now turn our attention to relating forest clusters with respect to $\sigma$ to $\sigma$-nice trees, restricting our attention to the cases $\sigma=1423$ and $\sigma=1324$. One important observation is that in any forest cluster with respect to $\sigma=1423$ or $\sigma=1324$, the root vertex has the smallest label. For visual reference, the ways that two consecutive instances of $\sigma$ can overlap for $\sigma=1423$ or $\sigma=1324$ are given in \cref{fig:sigma-overlap}.

\begin{definition}
Let the pattern $\sigma$ be either $1423$ or $1324$, and let $X$ be a forest cluster with respect to $\sigma$. A \textit{$\sigma$-nice subtree} $T$ of $X$ is an induced subgraph of $X$ forming a $\sigma$-nice tree whose root is the root of $X$, and such that every even-depth vertex of $T$ with depth at least $4$ is the endpoint of a consecutive instance of $\sigma$ that is highlighted in $X$. The union of all such $\sigma$-nice subtrees is another $\sigma$-nice subtree, which we refer to as the \textit{maximum $\sigma$-nice subtree}, denoted $T_\mathrm{max}(X)$.
\end{definition}
Note that in the above definition, since $T$ has the same root vertex as $X$, the depth of a vertex of $T$ is unambiguous. Also, it is clear that any forest cluster contains at least one $\sigma$-nice subtree, so that $T_\mathrm{max}(X)$ is well-defined.

\begin{figure}
    \centering
    \includegraphics{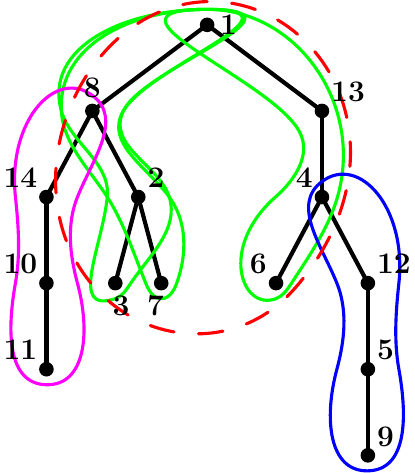}
    \caption{A $5$-forest cluster of size $14$ with respect to $1423$, with the maximum $1423$-nice subtree circled by the dashed line.}
    \label{fig:maximum-nice}
\end{figure}

An example is given in \cref{fig:maximum-nice}, which shows a $5$-forest cluster of size $14$ with respect to $\sigma=1423$. The maximum $\sigma$-nice subtree is circled by the dashed line, and consists of three highlighted consecutive instances of $\sigma$. If the consecutive instance corresponding to the labels $1, 13, 4, 12$ was also highlighted, making the cluster a $6$-forest cluster, then the maximum $\sigma$-nice subtree would also contain the vertices labeled $12, 5, 9$.

Clearly the root of $X$ and its children all lie in $T_\mathrm{max}(X)$. Furthermore, the $\sigma$-nice tree $T_\mathrm{max}(X)$ has no childless vertices of depth $2$.

\begin{definition}
Let the pattern $\sigma$ be either $1423$ or $1324$, and let $X$ be a forest cluster with respect to $\sigma$. Let \textit{$R(X)$} be the set of vertices of $T_\mathrm{max}(X)$ that have at least one child that does not lie in $T_\mathrm{max}(X)$.
\end{definition}

\begin{lemma}\label{lem:instances-intersecting-R(X)}
Let $\sigma$ be either $1423$ or $1324$, and let $X$ be a forest cluster with respect to $\sigma$. Let $v\in R(X)$, and let $u$ be a child of $v$ that does not lie in $T_\mathrm{max}(X)$. Then, there exists exactly one highlighted consecutive instance of $\sigma$ in $X$ that contains $v$ and $u$. Moreover, this consecutive instance has starting point $v$.
\end{lemma}

\begin{proof}
It follows from the definition of a forest cluster that there exists at least one highlighted consecutive instance of $\sigma$ that contains $v$ and $u$. Now suppose the sequence $v_1,v_2,v_3,v_4$ of vertices of $X$ is a highlighted consecutive instance of $\sigma$ such that $v_i=v$ and $v_{i+1}=u$ for some $i\in\{1,2,3\}$ (by the definition of a consecutive instance, $v_1$ is the parent of $v_2$, and so on). It suffices to show $v\neq v_2,v_3$.

Suppose first $v=v_3$. Then,the label of $v$ is less than the label of its parent $v_2$, and $v\in T_\mathrm{max}(X)$, so $v$ has odd depth. But we may then include $u$ in $T_\mathrm{max}(X)$, a contradiction.

Now if $v=v_2$ and $v$ has odd depth, then it must have a child $u'$ lying in $T_\mathrm{max}(X)$. The depth of $u'$ is even and at least $3$, so there exists a highlighted consecutive instance of $\sigma$ with $u'$ as its endpoint. But this implies the label of $v_2$ is less than the label of $v_1$, a contradiction. If $v=v_2$ and $v$ has even depth, then we may include $u$ and $v_4$ in $T_\mathrm{max}(X)$, again a contradiction.
\end{proof}

\begin{definition}
Let the pattern $\sigma$ be either $1423$ or $1324$, and let $X$ be a forest cluster with respect to $\sigma$. Let $v\in R(X)$, and let $u_1,\dots,u_s$ be the children of $v$ that do not lie in $T_\mathrm{max}(X)$. For $1\le j\le s$, let $D_{u_j}$ be the set of descendants of $u_j$ (which includes $u_j$ itself). Let \textit{$T_{X,v}$} be the subgraph induced by the set $\{v\}\cup(D_{u_1}\cup\cdots\cup D_{u_j})$. The graph $T_{X,v}$ forms a rooted tree with root $v$, which becomes a forest cluster (also denoted $T_{X,v}$) if we take the highlighted consecutive instances of $\sigma$ in $X$ whose vertices all lie in $T_{X,v}$.
\end{definition}
In the above definition, note that $s\ge 1$ since $v\in R(X)$. The claim that $T_{X,v}$ becomes a forest cluster follows from \cref{lem:instances-intersecting-R(X)}.

Informally, the maximal $\sigma$-nice subtree $T_\mathrm{max}(X)$ and the attached forest clusters $T_{X,v}$ for $v\in R(X)$ will form a ``decomposition'' of $X$ (unrelated to the \textit{decomposition} of a $\sigma$-nice forest that is defined in \cref{def:nice-decomposition}), where $R(X)$ is the set of roots of the attached clusters. These ideas are recorded in the following result:
\begin{lemma}\label{lem:forest-cluster-decomposition}
Let $\sigma$ be either $1423$ or $1324$, and let $X$ be a forest cluster with respect to $\sigma$. Let $R(X)=\{v_1,\dots,v_r\}$. Then, each vertex of $X$ lies in $T_\mathrm{max}(X)$ or in $T_{X,v}$ for some $v\in R(X)$, and for all $i\in\{1,\dots,r\}$, there exists exactly one vertex in both $T_\mathrm{max}(X)$ and $T_{X,v_i}$, namely $v_i$. For distinct $i,j\in\{1,\dots,r\}$, the forest clusters $T_{X,v_i}$ and $T_{X,v_j}$ do not share any vertices. Finally, each highlighted consecutive instance of $\sigma$ lies entirely in exactly one of $T_\mathrm{max}(X),T_{X,v_1},T_{X,v_2},\dots,T_{X,v_r}$.
\end{lemma}

\begin{proof}
The proof follows easily from \cref{lem:instances-intersecting-R(X)}.
\end{proof}

We are finally ready to prove \cref{thm:1324-1423}. We will use the ``decomposition'' of a forest cluster to write down a recurrence for the forest cluster numbers.

\begin{proof}[Proof of \cref{thm:1324-1423}]
By \cref{thm:same-cluster-numbers}, it is equivalent to show that $1423$ and $1324$ have equal forest cluster numbers. Let $\sigma$ be either $1423$ or $1324$. For a positive integer $n$, let $L_R$ be a subset of $\{2,\dots,n\}$, and let $r=|L_R|$. Suppose $L_R=\{\ell_1,\dots,\ell_r\}$, where $\ell_1<\cdots<\ell_r$. Let $L_1,\dots,L_r$ be disjoint subsets of $\{2,\dots,n\}$ such that $L_i\subseteq\{\ell_i,\ell_i+1,\dots,n\}$ and $L_R\cap L_i=\{\ell_i\}$ for all $1\le i\le r$. For simplicity, let $\mathcal{L}_n$ be the set of all pairs $(L_R,(L_1,\dots,L_r))$ satisfying these conditions.

Now given integers $n,m\ge 1$, sets $L_R,L_1,\dots,L_r$ such that $(L_R,(L_1,\dots,L_r))\in\mathcal{L}_n$, and nonnegative integers $m_0,m_1,\dots,m_r$ that sum to $m$, let $N_\sigma(n,m,L_R,(L_1,\dots,L_r),(m_0,\dots,m_r))$ be the number of $m$-forest clusters $X$ on $[n]$ with respect to $\sigma$ such that the following hold:
\begin{itemize}
    \item The tree $T_\mathrm{max}(X)$ has label set $([n]\setminus (L_1\cup\cdots\cup L_r))\cup L_R$.
    \item $|R(X)|=r$ and the labels of the vertices in $R(X)$ form the set $L_R$.
    \item For all $1\le i\le r$, the label set of the forest cluster $T_{X,v_i}$ is $L_i$, where $v_i\in R(X)$ denotes the vertex with label $\ell_i$.
    \item $T_\mathrm{max}(X)$ contains exactly $m_0$ highlighted consecutive instances.
    \item For all $1\le i\le r$, the forest cluster $T_{X,v_i}$ contains exactly $m_i$ highlighted consecutive instances.
\end{itemize}
Then using \cref{lem:forest-cluster-decomposition} and the fact that the root vertex has the smallest label in any forest cluster with respect to $\sigma$, it is clear that
\[r_{n,m}=\sum_{(L_R,(L_1,\dots,L_r))\in\mathcal{L}_n}\sum_{\substack{m_0,\dots,m_r\in\mathbb{Z}_{\ge 0}\\m_0+\cdots+m_r=m}}N_\sigma(n,m,L_R,(L_1,\dots,L_r),(m_0,\dots,m_r)).\]
We claim that
\[N_\sigma(n,m,L_R,(L_1,\dots,L_r),(m_0,\dots,m_r))=B_\sigma(n-(|L_1|+\cdots+|L_r|)+|L_R|,m_0)\prod_{i=1}^rr_{|L_i|,m_i}.\]
To see this, if $X$ is counted by $N_\sigma(n,m,L_R,(L_1,\dots,L_r),(m_0,\dots,m_r))$, then $T_\mathrm{max}(X)$ is a $\sigma$-nice tree with no childless vertices of depth $2$ and contains exactly $m_0$ highlighted consecutive instances of $\sigma$. Note also that every consecutive instance of $\sigma$ in $T_\mathrm{max}(X)$ is highlighted. In addition, for all $1\le i\le r$, the forest cluster $T_{X,v_i}$ consists of $|L_i|$ vertices and contains exactly $m_i$ highlighted consecutive instances (where as usual $v_i$ denotes the vertex with label $\ell_i$).

Conversely, suppose we are given a $\sigma$-nice tree $T$ with no childless depth-$2$ vertices with label set $([n]\setminus(L_1\cup\cdots\cup L_r))\cup L_R$ and containing exactly $m_0$ consecutive instances of $\sigma$, and for all $1\le i\le r$ we are given a forest cluster $T_i$ with label set $L_i$ and containing $m_i$ highlighted consecutive instances of $\sigma$. If for all $i$ we attach the cluster $T_i$ to the vertex $v_i$ labeled $\ell_i$ in $T$ and highlight all consecutive instances contained entirely in $T$, then we obtain a forest cluster $X$ with $n$ vertices and $m$ highlighted consecutive instances of $\sigma$ such that $T_\mathrm{max}(X)=T$ and $T_{X,v_i}=T_i$ for all $i$. The claimed equality follows.

Thus we find the recurrence
\[r_{n,m}=\sum_{(L_R,(L_1,\dots,L_r))\in\mathcal{L}_n}\sum_{\substack{m_0,\dots,m_r\in\mathbb{Z}_{\ge 0}\\m_0+\cdots+m_r=m}}B_\sigma(n-(|L_1|+\cdots+|L_r|)+|L_R|,m_0)\prod_{i=1}^rr_{|L_i|,m_i}.\]
The result then follows by inducting on $n$, with base case $r_{1,m}=0$, and using \cref{prop:B1423=B1324}.
\end{proof}

We arrive at the following surprising result:

\begin{corollary}\label{cor:cfw-and-cw}
In general, c-forest-Wilf equivalence does not imply c-Wilf equivalence.
\end{corollary}

As mentioned in the introduction, it is known that the patterns  $1423$ and $1324$ are not c-Wilf equivalent \cite[Table~3]{elizaldenoy1}.

\subsubsection{Enumerating $\sigma$-extranice trees for $\sigma=1234,1423,1324$}\label{subsubsec:sigma-extranice-enumeration}
Given our work, enumerating $\sigma$-extranice forests is not as difficult as analogous problems in pattern avoidance in rooted forests. We first prove \cref{thm:extranice-equal}, part of which states that the numbers of $\sigma$-extranice trees are equal for $\sigma=1234,1423,1324$.

\begin{definition}
Given a tree $T$ and an odd-depth vertex $v$ of $T$, we define a map $g_v$ that permutes the labels of the subtree rooted at $v$, while preserving the underlying structure of $T$.
\begin{enumerate}
    \item Let the label set of the subtree rooted at $v$ be $L$, and let $j\in L$ be the label of $v$. Relabel every child $w$ of $v$ so that if $w$ originally has the $i$th smallest label in $L\setminus \{j\}$, then $w$ is relabeled to have the $i$th largest label in $L \setminus \{j\}$.
    \item Relabel each subtree $X$ rooted at a child of a child of $v$ as follows. Suppose $X$ originally has label set $D$. Then relabel $X$ using the label set $E$ while keeping all labels of vertices of $X$ in the same relative order, where $E$ is defined by replacing each element $x$ of $D$ by the element $y$ of $L \setminus \{j\}$ such that if $x$ is the $i$th smallest element of $L \setminus \{j\}$, then $y$ is the $i$th largest element of $L \setminus \{j\}$. Note that this relabeling is analogous to the $\alpha_E$ operation on twigs.
\end{enumerate}
The map $G$ relabels a tree $T$ by applying $g_v$ for all odd-depth vertices $v$; for $i\ge 1$, in Step $i$ of the construction of $G(T)$ from $T$, the map $g_v$ is applied to the tree for each vertex $v$ of depth $2i-1$. The steps are performed numerical order, i.e., first Step $1$, then Step $2$, and so on.
\end{definition}

\begin{lemma}\label{lem:G-change}
If $T$ is $1423$-nice, then $G(T)$ is $1234$-nice. Similarly, if $T$ is $1234$-nice, then $G(T)$ is $1423$-nice.
\end{lemma}
\begin{proof}
Suppose $T$ is $1423$-nice, and suppose we apply $G$ to $T$. We claim that immediately after Step $i$, every vertex of depth $2i+2$ is the endpoint of a consecutive instance of $1234$, while every subtree rooted at a vertex of depth $2i+1$ remains $1423$-nice. We will prove this by induction on $i$. For the base case $i=1$, let $v$ be the root of $T$. Then since $T$ is $1423$-nice, every child $w$ of $v$ has a label greater than the label of each of the strict descendants of $w$. Thus when applying $g_v$, the vertex $w$ ends up with a label less than the label of each of its strict descendants, but still greater than the label of $v$, so every depth-$4$ vertex is now the endpoint of a consecutive instance of $1234$. Meanwhile, each subtree rooted at a child of $w$ has the labels of its vertices in the same relative order, so each such subtree remains $1423$-nice. This finishes the base case. The inductive step is addressed similarly, finishing the proof of the claim.

Note that for all $i\ge 1$, immediately after Step $i$, each vertex $v$ with depth $2i-1$ still has a smaller label than each of its children, and if $i\ge 2$, this implies each child of $v$ is still an endpoint of a consecutive instance of $1234$. These statements continue to hold throughout the remainder of the construction of $G(T)$. Combining this with the claim, we see that $G(T)$ must be $1234$-nice.

The proof that $G(T)$ is $1423$-nice if $T$ is $1234$-nice is similar.
\end{proof}
\begin{lemma}\label{lem:G-involution}
For any tree $T$, we have $G(G(T))=T$.
\end{lemma}
\begin{proof}
First, it is easy to see that $g_v$ itself is an involution. Now we claim that $g_v, g_w$ commute for any two odd-depth vertices $v, w$. If neither one of $v, w$ is an ancestor of the other, or if $v=w$, then this is trivially true. Otherwise, suppose without loss of generality that $v$ is a strict ancestor of $w$. Let $X$ be the subtree rooted at $w$. Note that each vertex of $T$ that does not lie in $X$ has the same label in $g_v(g_w(T))$ and in $g_w(g_v(T))$, so we only need to consider the vertices of $X$. Applying $g_v$ does not change the relative order of the labels of the vertices of $X$, and changes the label set of $X$ based on its original label set. The map $g_w$ does not change the label set of $X$, and changes the relative order of the labels of the vertices of $X$ in a manner independent of the specific label set used, and only depending on the original relative order of the labels of the vertices of $X$. Thus $g_vg_w=g_wg_v$, proving the claim. Since $G$ is a composition of commuting involutions, it is also an involution.
\end{proof}

\begin{theorem}\label{thm:extranice-equal}
Fix $n$. The numbers of $\sigma$-nice forests (resp.\ trees) on $[n]$ are equal for $\sigma=1234, 1423, 1324$. Furthermore, the numbers of $\sigma$-extranice forests (resp.\ trees) on $[n]$ are also equal for $\sigma=1234, 1423, 1324$.
\end{theorem}
\begin{remark}
Note that the statement of \cref{thm:extranice-equal} for $\sigma$-nice (resp.\ $\sigma$-extranice) forests is automatically equivalent to the statement for $\sigma$-nice (resp.\ $\sigma$-extranice) trees, though we will not use this in the proof.
\end{remark}
\begin{proof}
We first prove all four statements for $\sigma=1423$ and $\sigma=1324$. The statement for $\sigma$-nice trees follows directly from \cref{prop:A1423=A1324}, and the statement for $\sigma$-nice forests follows from a straightforward modification of the proof of \cref{prop:A1423=A1324}. Note that a $\sigma$-nice forest (resp.\ tree) is $\sigma$-extranice if and only if its decomposition is a proper twig collection in which each twig has exactly one child vertex. So, the number of $\sigma$-extranice forests (resp.\ trees) on $[n]$ equals the sum of $F_\sigma(W)$ (resp.\ $T_\sigma(W)$) over all proper twig collections $W$ with label set $[n]$ in which each twig has exactly one child vertex. Thus, we are done by \cref{lem:gamma-child-vertex-labels-permutation,prop:gamma-involution,prop:twig}.

By \cref{lem:G-change} and \cref{lem:G-involution}, $G$ is a structure-preserving bijection between $1423$-nice trees and $1234$-nice trees. So, the numbers of $\sigma$-nice forests (resp.\ trees) on $[n]$ are equal for $\sigma=1234, 1423$. Whether a $\sigma$-nice forest (resp.\ tree) is $\sigma$-extranice depends only on the structure of the forest (resp.\ tree), and therefore the numbers of $\sigma$-extranice forests (resp.\ trees) on $[n]$ are equal for $\sigma=1234, 1423$. Thus, the theorem is proved.
\end{proof}

\begin{proposition}
The number of $\sigma$-extranice trees on $[2n]$ equals
\[\frac{(-1)^{n-1}2^{n+1}(2^{2n}-1)B_{2n}}{2n},\]
where $B_{2n}$ is the $(2n)$th Bernoulli number, for $\sigma=1234,1423,1324$.
\end{proposition}

\begin{proof}
By \cref{thm:extranice-equal}, it is enough to show the formula holds for $\sigma=1423$. The root of a $1423$-extranice tree on $[n]$ must be labeled $1$, and its unique child must be labeled $n$; below this vertex, we have a $1423$-extranice forest on $n-2$ vertices. So, letting $F(n)$ (resp.\ $T(n)$) be the number of $1423$-extranice forests (resp.\ trees) on $[n]$, we have $T(n)=F(n-2)$. If we use the exponential generating functions $f=\sum_{n=0}^{\infty} F(n)x^n/n!$ and $t=\sum_{n=0}^{\infty} T(n)x^n/n!$, then as in \cref{sec:treeforest}, by \cref{eq:forest-from-tree}, we have $t=\log(f)$. In addition, we have the new relation $f=t''$. So, $t=\log(t'')$, giving $t' \cdot t'' = t'''$. Using $T(0)=1, T(1)=0, T(2)=1$, we can compute \[t=1-2\log\cos\left(\frac{x}{\sqrt{2}}\right).\]

A more recognizable form is $t'=\sqrt{2}\tan(x/\sqrt{2})$. Using the formula for the tangent Maclaurin series and integrating, we have 
\[T(2n)=\frac{(-1)^{n-1}2^{n+1}(2^{2n}-1)B_{2n}}{2n},\]
so we are done. This formula was found using computer calculations.
\end{proof}

\section{Future work}
We have more conjectured forest-Wilf equivalences.

\begin{conjecture}
The following three forest-Wilf equivalences hold, where we use $\sim$ to denote forest-Wilf equivalence:

\begin{itemize}
            \item $\{123,2413\}\sim\{132,2314\}$.
            \item $\{123,3142\}\sim\{132,3124\}$.
            \item $\{213,4123\}\sim\{213,4132\}$.
        \end{itemize}
\end{conjecture}

This conjecture has recently been proven in \cite{ren}.

We also have the following conjectures for the asymptotic behavior of these counts.

\begin{conjecture}
For any set of patterns $S$, there exists a constant $C_S\ge 0$ such that
\[ \lim_{n \rightarrow \infty} \frac{f_n(S)^{\frac{1}{n}}}{n}=\lim_{n \rightarrow \infty}\frac{t_n(S)^{\frac{1}{n}}}{n} = C_S.\]
Furthermore, if $C_S>0$, then each of the sequences $(f_{n+1}(S)/f_n(S)-f_n(S)/f_{n-1}(S))_{n \ge 1}$ and $(t_{n+1}(S)/t_n(S)-t_n(S)/t_{n-1}(S))_{n \ge 2}$ is monotonic for sufficiently large $n$ and converges to $eC_S$.
\end{conjecture}

Here $e$ denotes Euler's constant. Based on data, we conjecture $C_{213, 231} \approx 0.557864$, $C_{213} \approx 0.65521$, $C_{213, 321} \approx 0.5530$, $C_{213, 123} \approx 0.555843$, and $C_{123} \approx 0.6801$.

The recent paper \cite{ren} also makes progress on this conjecture.

We also make the following general conjecture.

\begin{conjecture}
If the two sets $S$ and $S'$ of patterns are forest-Wilf equivalent, then $S$ and $S'$ are Wilf equivalent (with respect to pattern avoidance in permutations).
\end{conjecture}

One may make the stronger conjecture that $S$ and $S'$ are even forest-structure-Wilf equivalent. By \cref{cor:cfw-and-cw}, the analogous statements for c-Wilf equivalence do not hold. So far, all instances of forest-Wilf equivalence we have found imply Wilf equivalence and forest-structure-Wilf equivalence.

We also wonder whether there are bijections between forests (rather than just clusters) in the consecutive pattern case, similar to in the non-consecutive case.

\begin{question}
Is there an explicit bijection between forests with $m$ consecutive instances of $1324$ and $m$ consecutive instances of $1423$, for general $m$?
\end{question}

Finally, we ask the following question.
\begin{question}
Are there nontrivial single-pattern c-forest-Wilf equivalences other than the one between $1324$ and $1423$?
\end{question}
Using a computer, we have verified that no other nontrivial c-forest-Wilf equivalences exist up to patterns of length $5$. This question has since been answered in the affirmative by \cite{ren}.

\section{Acknowledgments}
This research was funded by NSF/DMS grant 1659047 and NSA grant H98230-18-1-0010. The authors would like to thank Prof.\ Joe Gallian for organizing the Duluth REU where this research began and suggesting the topic of research, as well as advisors Aaron Berger and Colin Defant. We would also like to thank Ashwin Sah, Mehtaab Sawhney, Shyam Narayanan, and Defant for their helpful suggestions in discussing possible directions of research, and Amanda Burcroff, Defant, Gallian, and Sah for providing helpful comments on the initial drafts. We also thank Peter Luschny for his discussions regarding the Bell transform. The authors especially acknowledge Sah's extensive advice. We also thank the anonymous reviewer for their detailed and helpful comments.

\bibliographystyle{abbrv}
\bibliography{main}

\end{document}